\documentclass[12pt]{article}
\usepackage{dsfont}
\usepackage[english]{babel}
\usepackage{ucs}
\usepackage[utf8x]{inputenc}
\usepackage{amsthm}
\usepackage{enumerate}
\usepackage{amsmath}
\usepackage [top=3cm, bottom=3cm, left=3cm, right=3cm] {geometry}
\usepackage{graphicx}
\usepackage{amssymb}
\usepackage{hyperref}
\usepackage{todonotes}
\usepackage{setspace}

\newtheorem{theorem}{Theorem}[section]
\newtheorem{lemma}[theorem]{Lemma}
\newtheorem{proposition}[theorem]{Proposition}
\newtheorem{corollary}[theorem]{Corollary}

\theoremstyle{definition}
\newtheorem{definition}[theorem]{Definition}

\theoremstyle{remark}
\newtheorem{remark}[theorem]{Remark}
\numberwithin{equation}{section}
\usepackage{color}
\usepackage[babel=true]{csquotes}

\newcommand{\beq}{\begin{equation} \begin{split}}
\newcommand{\feq}{\end{split} \end{equation}}
\definecolor{GRIS}{cmyk}{0.7,0.6,0.5,0.3}
\definecolor{BLEU}{cmyk}{1,0.9,0.1,0}
\definecolor{GM-ASI}{cmyk}{0.6,0,0,0}
\definecolor{EP-MEKA}{cmyk}{0.3,1,0,0}
\definecolor{CP-MRIE}{cmyk}{0.3,0,1,0}
\definecolor{GCCD}{cmyk}{0,0.5,1,0}
\usepackage{authblk}

        \title{Solving stochastic differential equations with Cartan's exterior differential system}
\author[1]{Paul Lescot \footnote{palescot@gmail.com.}}
  \author[1]{Helene Quintard\footnote{helene.quintard@gmail.com}}
  \author[2]{Jean-Claude Zambrini \footnote{zambrini@cii.fc.ul.pt} }
  \affil[1]{{\small{Normandie Universit\'e, Universit\'e de Rouen, \newline
   Laboratoire de Math\'ematiques Rapha\"el Salem, CNRS,UMR 6085, \newline
   Avenue de l'universit\'e, BP 12, 76801 Saint-Etienne du Rouvray Cedex, \newline
    France.}}}
  \affil[2]{\small{GFM,Faculdade de Ci\^encias, Universidade de Lisboa, \newline Av. Gama Pinto 2 
1649-003 , Lisboa \newline
Portugal.}}

\begin{document}
\maketitle
\begin{abstract}
 The aim of this work is to use systematically the symmetries of the (one dimensional) PDE : 
\begin{equation}
\gamma \dfrac{\partial \eta}{\partial t}= -\dfrac{\gamma^2}{2} \dfrac{\partial^2 \eta}{\partial q^2}+V(t,q)\eta \equiv H\eta \quad , \label{eqdep}
\end{equation}
where $\gamma$ is a positive constant and $V$ a given \enquote{potential}, in order to solve one dimensional Itô's  stochastic differential equations (SDE) of the form :
\begin{equation}
dz(t)=\sqrt{\gamma}dw+ \gamma \dfrac{\partial}{\partial q} \ln \eta(t, z(t)) dt
\end{equation}
where $w$ denotes a Wiener process. The special form of the drift of $z$ (suggested by quantum mechanical considerations) gives, indeed, access to an algebrico-geometric method due, in essence, to E.Cartan, and called the Method of Isovectors. 
A $V$ singular at the origin, as well as a one-factor affine model relevant to stochastic finance, are considered as illustrations of the method.
\end{abstract}




\section{Introduction}

Let us consider a one-dimensional Itô's stochastic differential equation (SDE) for a diffusion process $z(\cdot)$ , of the form 
\begin{equation}
dz(t)=\sqrt{\gamma}dw (t) +\widetilde{B}(t,z(t)) dt \label{Intro.defz}
\end{equation}
where $\gamma$ denotes a positive constant, $w$ a standard Wiener process (or Brownian motion) and $D$ is the drift of the process. The reason why it is called a drift is that $\widetilde{B}$ represents a (conditioned) mean time derivative. More generally if $F=F(t,q)\in C^2_0$ then it is in the domain $\mathcal{D}_D$ of the following \enquote{infinitesimal generator} of $z$ :
\begin{equation}
\begin{split}
D_t F(t,q) =& \left(\dfrac{\partial}{\partial t} +\widetilde{B} \dfrac{\partial}{\partial q} + \gamma \dfrac{\partial^2}{\partial q^2} \right)F(t,q)\\
=&\lim_{ \Delta t\searrow 0} E_t \left[\dfrac{F(t+\Delta t,z(t+\Delta t))-F(t,z(t))}{\Delta t}\right] \label{Intro.generator}
\end{split}
\end{equation}
where $E_t$ denotes the conditional expectation given $z(t)=q$. Eq(\ref{Intro.generator}) results from K.Itô's formula. The theory of equations like Eq(\ref{Intro.defz}) is due to him and has been developed along the same line as the corresponding deterministic first order ODE (ie with $\gamma=0$ \cite{R1}). In particular, it is known that if $\widetilde{B}$ satisfies a Lipschitz condition, $z_s$ is mesurable with respect to a $\sigma$-algbra $\mathcal{P}_s$ and such that, if $E\vert z_s\vert^2 < \infty $, then Eq(\ref{Intro.defz}) with initial condition $z(s)=z_s$ has an unique solution $z(t)$, $t \geqslant s$.\\
As it is familiar in the deterministic context of ODE, the richest mathematical structures appear for equations of second order in time, those introduced historically for Newton systems of equations \cite{R3}. The Lagrangian and Hamiltonian frameworks are indeed at the origin of the modern algebraic, geometric and analytic approaches to dynamics. \\
When the only information we have on $\widetilde{B}$ is, as before, a Lipschitz condition there is no way to connect Eq(\ref{Intro.defz}) with any such frameworks. On the other hand, there is a general approach called \enquote{Stochastic Deformation} \cite{R2} whose aim is precisely to deform along the paths of various kinds of stochastic processes the main geometric and algebraic tools of classical mechanics. In this context, it becomes clear that the drift $\widetilde{B}$ of Eq(\ref{Intro.defz}) needs to be of very special forms. For the elementary partial operator $H$ mentioned in the abstract, $\widetilde{B}$ should be a logarithmic derivative (for a general $H$, however, $\widetilde{B}$ is not a gradient) then it is possible to implement $H$ as playing the same role as a quantum Hamiltonian deforming a Newtonian system with potential V. The main difference with the quantization method is that, although the quantum deformation is not a probabilistic one, the deformation advocated here uses well defined probability measures, in our elementary diffusion case defined on spaces of continuous trajectories. The reason why for $H$ as before, $\widetilde{B}$ should be a logarithmic gradient derivative is, in fact, of quantum origin.
Let $\langle A \rangle_{\phi_t}$ denotes the quantum \enquote{expectation} of any observable $A$ (for the system with Hamiltonian $H$) in the state  $\phi_t$ at time $t$, $\phi_t \in L^2(\mathbb{R})$. Then denoting by $P$ the momentum of this system,  
\begin{equation}
\begin{split}
\langle P \rangle_{\phi_t}=&\int {\phi}_t(q) (-i\sqrt{\gamma} \bigtriangledown \phi_t(q)) dq\\
=&\int \vert {\phi}_t(q)\vert^2  \bigtriangledown (-i\sqrt{\gamma} \log \phi_t(q)) dq 
\end{split}
\end{equation}
So to the observable $P$ sould correspond $\bigtriangledown (-i\sqrt{\gamma} \log\phi_t(q))$, where $q$ should be the value of a random variable with probability density $\vert {\phi}_t(q)\vert^2 dq$.\\
Here we are going to stay as close as possible from this quantum context, but definitely on the probabilistic side where probabilistic processes make sense.\\
The content of this paper is the following :\\
In the section \ref{HJB}, we describe the stochastic deformation of the Hamilton-Jacobi equation for our underlying Newtonian system. It is known in the litterature of optimal control theory as Hamilton-Jacobi-Bellman (HJB). We summarize, then, the method of isovectors and its results for HJB, following \cite{LZ2}.\\
Section 3 is devoted to an algebraic approach to the deformation of the classical Liouville $2$-form $\Omega$, the foundation of Symplectic Geometry.\\
In relation with classical Hamiltonian mechanics, a key property of $\Omega$ is to be invariant under the time evolution. We shall prove the stochastic deformation of this property in Section 4.
The rest of the paper, divided into various sections, considers in details two illustrations of the method of isovectors for solving stochastic differential equations. The first one focuses on the symmetries for a potential $V$ in $H$ of the form $V= \dfrac{C}{q^2}+Dq^2$ and any values of the constants $C$ and $D$. The second describes an alternative approach to an affine interest rate model \cite{Henon, LeblancScaillet, theseHouda}. As a matter of fact, the above potential $V$ appears in analysis of this model, justifying its careful study in terms of $C$ and $D$.\\
The conclusion will explain what are the main orientations and motivations of this unusual approach to stochastic differential equation and what we are expecting from it. In particular, its potential relations with a general concept of stochastic integrability for SDE, still missing, will be described. Of course, independently of the motivations of the program of stochastic deformation, we hope that the methods used here will be of some interest in the other applied fields. For instance, in stochastic finance, the idea to introduce methods coming from quantum physics, cf, for instance \cite{BB} or \cite{EEH} to mention only those, is not really new. Also the use of Lie symmetry methods to compute some expectations of diffusion processes is already known (cf ,for example,  \cite{CrLe}). But it does not seem that quantum-like symmetries have really been exploited,up to now,in the study of solutions of stochastic differential equations,in a dynamical perspective like the one advocated here.

\section{Isovectors of Hamilton-Jacobi-Bellman}
\label{HJB}
For a smooth potential $V=V(t,q)$, let us consider the PDE 
\begin{equation}
\gamma \dfrac{\partial \eta}{\partial t}= -\dfrac{\gamma^2}{2} \dfrac{\partial^2 \eta}{\partial q^2}+V(t,q)\eta \equiv H\eta \quad . \label{HJB.edpeta}
\end{equation}
If a solution $\eta$ exists and is positive (typically on a finite time interval) then,
\begin{equation}
S(t,q)=-\gamma \ln \eta(t,q) \label{HJB.defS}
\end{equation}
 solves Hamilton-Jacobi-Bellman (HJB) equation 
 \begin{equation}
 -\dfrac{\partial S}{\partial t} + \dfrac{1}{2} \left(\dfrac{\partial S}{\partial q}\right)^2-V- \dfrac{\gamma}{2} \dfrac{\partial^2 S}{\partial q^2}=0 \label{HJB.EDPS}.
 \end{equation}
 
Eq(\ref{HJB.EDPS}) is interpreted as a stochastic deformation of (one of the two adjoint \cite{CH-Z}) Hamilton-Jacobi equation for the classical system underlying the \enquote{Hamiltonian} $H$.\\
It will be convenient to introduce partial derivatives of Eq(\ref{HJB.EDPS}) as new variables
\begin{equation}
B=-\dfrac{\partial S}{\partial q} \qquad , \qquad \qquad E=-\dfrac{\partial S}{\partial t} \label{HJB.defBE}.
\end{equation}

Cartan has shown (cf \cite{HE71}) that the geometric interpretation of HJB Eq(\ref{HJB.EDPS}) in the $1$-jet $J^1$ of independant variables $(t,q,S,E,B)$ corresponds to the vanishing of the differential forms (where PC stands for Poincar\'e-Cartan)
\begin{equation}
\omega=Bdq+Edt+dS\equiv \omega_{PC} +dS \label{HJB.defomegaPC},
\end{equation}

which is a contact $1$-form, of its exterior derivative,
\begin{equation}
\Omega=d\omega=dBdq+dEdt \label{HJB.Omega},
\end{equation}

and of a two form $\beta$ reproducing HJB itself, namely,
\begin{equation}
\beta=(E+\dfrac{1}{2}B^2-V)dqdt+\dfrac{\gamma}{2} dBdt \label{HJB.beta}.
\end{equation}

Our probabilistic interpretation will hold, in fact, not on $J^1$ but on the section of $S$ on the base (\enquote{integral}) submanifold $(t,q)$, i.e. $(t,q,S(t,q),-\dfrac{\partial S}{\partial t}(t,q),-\dfrac{\partial S}{\partial q}(t,q))$.

Indeed, it is the $q$ variable which will be later on promoted to the status of the random variable $z(t)$ solving Eq(\ref{Intro.defz}). The rest will follow from appliying Itô's calculus to the \enquote{Exterior differential system} $(\omega, \Omega, \beta)$ properly interpreted along the paths $t\mapsto z(t)$.\\
Those three forms generate a \enquote{differential} ideal $I$. This means that for any form in $I$, its exterior derivative is also in $I$. \\
A vector field $N$ on $J^1$, $N=+N^q \dfrac{\partial}{\partial q}+N^t \dfrac{\partial}{\partial t}+N^S \dfrac{\partial}{\partial S}+N^B \dfrac{\partial}{\partial B}+N^E \dfrac{\partial}{\partial E}$ such that  its Lie dragging $\mathcal{L}_N$ leaves the ideal invariant, i.e., $\mathcal{L}_N(I) \subset I$ has been called an isovector in \cite{HE71}. The theory of isovectors was also inspired by Cartan. We shall call $\mathcal{G}_V$ their Lie algebra for Eq(\ref{HJB.EDPS}); of special importance will be its Lie subalgebra (cf [8], p.214): \\ $\mathcal{H}_V=\left\lbrace N \in \mathcal{G}_V \text{ such that } \dfrac{\partial N^S}{\partial S}=0 \right\rbrace$.

It follows from the proof of Theorem (3.1) in \cite{LZ2} that the isovectors of Hamilton-Jacobi-Bellman take the form :
\begin{equation}
\begin{split}
N^q&=\frac{1}{2}T'_N(t)q+l(t),\\
N^t&=T_N(t),\\
N^S&=h(t,q,S),\\
N^B&=\frac{1}{2}T'_N(t)B-\frac{\partial h}{\partial q}+B \frac{\partial h}{\partial S},\\
N^E&=-(\frac{1}{2}T''_N(t)q+l'(t))B-T'_N(t)E-\frac{\partial h}{\partial t}+E\frac{\partial h}{\partial S}, \label{HJB.defISO}
\end{split}
\end{equation}
where $h=h(t,q,S)$ can be expressed as

\begin{equation}
\begin{split}
&h(t,q,S)=\gamma p(t,q)e^{\frac{1}{\gamma}S}-\phi(t,q),\\
\label{HJB.conditiononh}
\end{split}
\end{equation}
and $p$, $\phi$, $l$ such that 
\begin{equation}
\begin{split}
&\phi(t,q)=\frac{1}{4}T_N''q^2+l'q-\sigma(t),\\
&\gamma \frac{\partial p}{\partial t}=-\dfrac{\gamma^2}{2}\frac{\partial^2 p}{\partial q^2}+pV,\\
&-\frac{\partial \phi}{\partial t}+T_N \frac{\partial V}{\partial t}+(\frac{1}{2} T_N'q+l)\frac{\partial V}{\partial q}-\dfrac{\gamma}{2}\frac{\partial^2 \phi}{\partial q^2}+T_N'V=0.
\label{HJB.conditiononpphil}
\end{split}
\end{equation}

If $\mu$ denotes the parameter of the transformation generated by the isovector $N$, we have :

\begin{equation}
e^{\mu N} :(t,q,S,E,B) \mapsto (t_\mu, q_\mu ,S_\mu ,E_\mu ,B_\mu).
\end{equation}

Following Eq(\ref{HJB.defS}) we shall introduce as well 
\begin{equation}
\eta_{\mu}(t_\mu,q_\mu)=e^{-\frac{1}{\gamma}S_\mu} \label{HJB.etamu}.
\end{equation}

By definition of the symmetry group of Eq(\ref{HJB.edpeta}) (cf \cite{OPJ}), $\eta_\mu$ solves the same equation but in the new coordinates $(t_\mu,q_\mu)$.\\
It is therefore natural to denote as well the one parameter group tranforming $\eta$ into $\eta_\mu$ by $e^{\mu \widetilde{N}}$, for some generator $\widetilde{N}$. When $N$ is as before, one checks easily that :

\begin{equation}
\widetilde{N}=-N^q \dfrac{\partial}{\partial q}-N^t \dfrac{\partial}{\partial t}-\dfrac{1}{\gamma}N^S \label{HJB.defNtilde}
\end{equation}

(There is a sign mistake in Eq(3.33) of \cite{LZ2} regarding the third component of $\widetilde{N}$.)

\begin{lemma} 
\label{HJB.lemmeNtilde}
The relation $N \mapsto - \widetilde{N}$ is a linear map compatible with the commutator, i.e. a Lie Algebra morphism.
\end{lemma}

 \begin{proof}
\begin{equation}
\begin{split}
 \widetilde{N}=&-N^t\frac{\partial }{\partial t}-N^q\frac{\partial }{\partial q}-\dfrac{1}{\gamma}N^S\\
\widetilde{M}=&-M^t\frac{\partial }{\partial t}-M^q\frac{\partial }{\partial q}-\dfrac{1}{\gamma} M^S\\
[\widetilde{N},\widetilde{M}]=&(N^t\frac{\partial M^t}{\partial t}-M^t\frac{\partial N^t}{\partial t}+N^q\frac{\partial M^t}{\partial q}-M^q\frac{\partial N^t}{\partial q})\frac{\partial}{\partial t}\\
& +(N^t\frac{\partial M^q}{\partial t}-M^t\frac{\partial N^q}{\partial t}+N^q\frac{\partial M^q}{\partial q}-M^q\frac{\partial N^q}{\partial q})\frac{\partial}{\partial q}\\
& +\dfrac{1}{\gamma}(N^t\frac{\partial M^S}{\partial t}-M^t\frac{\partial N^S}{\partial t}+N^q\frac{\partial M^S}{\partial q}-M^q\frac{\partial N^S}{\partial q})
\end{split}
\end{equation}

This expression can be simplified, because according to Eq(\ref{HJB.defISO}) $N^t$ is always independent from $q$ :

\begin{equation}
\begin{split}
[\widetilde{N},\widetilde{M}]=&(N^t\frac{\partial M^t}{\partial t}-M^t\frac{\partial N^t}{\partial t})\frac{\partial}{\partial t}\\
& +(N^t\frac{\partial M^q}{\partial t}-M^t\frac{\partial N^q}{\partial t}+N^q\frac{\partial M^q}{\partial q}-M^q\frac{\partial N^q}{\partial q})\frac{\partial}{\partial q}\\
& +\dfrac{1}{\gamma}(N^t\frac{\partial M^S}{\partial t}-M^t\frac{\partial N^S}{\partial t}+N^q\frac{\partial M^S}{\partial q}-M^q\frac{\partial N^S}{\partial q}).
\end{split}
\end{equation}
Let's show that $N\mapsto -\widetilde{N}$ is an algebra morphism, that is 
\begin{equation}
-\widetilde{[N,M]}=[-\widetilde{N},-\widetilde{M}]
\end{equation}
\begin{equation}
\label{demmorphisme}
\begin{split}
\widetilde{[N,M]}&=-[N,M]^t\dfrac{\partial}{\partial t}-[N,M]^q\dfrac{\partial}{\partial q}-\dfrac{1}{\gamma}[N,M]^S\\
[N,M]^t&=N(M^t)-M(N^t)\\
&=N^t\dfrac{\partial M^t}{\partial t}-M^t\dfrac{\partial N^t}{\partial t}\\
[N,M]^q&=N^t\frac{\partial M^q}{\partial t}+N^q\frac{\partial M^q}{\partial q}-M^t\frac{\partial N^q}{\partial t}-M^q\frac{\partial N^q}{\partial q}\\
[N,M]^S&=N^t\frac{\partial M^S}{\partial t}+N^q\frac{\partial M^S}{\partial q}-M^t\frac{\partial N^S}{\partial t}-M^q\frac{\partial N^S}{\partial q}\\
\end{split}
\end{equation}
Therefore  $-\widetilde{[N,M]}=[-\widetilde{N},-\widetilde{M}]$, so the mapping is a algebra morphism.
\end{proof}

Denoting by $\widetilde{\mathcal{H}}_V$ the set $\lbrace \widetilde{N} \text{ such that } N \in \mathcal{H}_V \rbrace$ it is clear that it is a Lie algebra, whose dimension depends on $V$. The associated local Lie group is precisely the above mentioned symmetry group of Eq (\ref{HJB.edpeta}).

\section{An algebraic approach to the symplectic $2$-form $\Omega$}

Let us define $\mathcal{J}_V$ by 
\begin{equation}
\mathcal{J}_V= \{N\in \mathcal{G}_V  \text{ such that }  N^S=c \text{ for } c\in \mathbb{R}\}
\end{equation} 
When the potential $V=0$, $\mathcal{J}_V$ was denoted by $\mathcal{H}_2$ in \cite{LZ1}.

\begin{lemma}
\label{Jv-Kv}
$\mathcal{J}_V$ is a Lie subalgebra of $\mathcal{G}_V$.
\end{lemma}

\begin{proof}
Using the same notation as for the definition of an isovector $N$,
$\forall(N, N') \in \mathcal{J}_V^2$ : 
\begin{equation}
\begin{split}
[N,N']^S &=[N,N'](S)\\
&= N(N'(S))-N'(N(S))\\
&=N(N'^S)-N'(N^S).
\end{split}
\end{equation}

This means that if $N$ and $N'$ are in $\mathcal{J}_V$, $\left[ N,N' \right]^S=0$ i.e. $\left[ N,N' \right]\in \mathcal{J}_V$, so $\mathcal{J}_V$ is a subalgebra of $\mathcal{H}_V$, hence of $\mathcal{G}_V$
\end{proof} 

Now we are going to need a \enquote{section map} $\theta_\eta$ describing what happen on the integral $2$-submanifold $(t,q)$ where our probabilistic interpretation will hold. This requires a positive $\gamma$ (cf Eq(\ref{Intro.defz})), which is a backward parabolic equation. As an initial value problem, it is known that, in general, it has no solution. However, where the solution exists, it is unique \cite{Fried}. Typically, given a very smooth initial condition, a classical solution will exist only up to a finite time $T$.\\
Since we are interested here in algebraic aspects, the fact that Eq(\ref{eqdep}) has, generally, only local solutions will not be a problem ; our results will be valid on their interval of existence. In fact, the special class of diffusions underlying our construction (\enquote{Bernstein} or \enquote{reciprocal} diffusions) are defined, by construction, on finite but arbitrary time intervals, so providing the solutions of a new kind of probabilistic boundary value problems (Cf \cite{R2} for more).\\ 

The logarithmic transform of variable in Eq(\ref{HJB.defS}) suggests the following definition.

\begin{definition}
Let $\eta$ be a positive solution of Eq(\ref{eqdep}). There is a unique differential algebra morphism. 
$\theta_\eta : \bigwedge T^* \mathbb{R}^5 \longrightarrow \bigwedge T^* \mathbb{R}^2$, where $\bigwedge T^*M$ denotes the bundle of differential forms, such that, $\theta_\eta(t)=t$, $\theta_\eta(q)=q$, $\theta_\eta(S)=-\gamma \ln (\eta) $, $\theta_\eta(E)=\dfrac{\gamma}{\eta}\dfrac{\partial \eta}{\partial t}$ and $\theta_\eta(B)=\dfrac{\gamma}{\eta}\dfrac{\partial \eta}{\partial q}$.

\end{definition}

The differential ideal $I\subset \ker \theta_\eta$ since the contact form $\omega$ and $\beta$ are in $\ker \theta_\eta$.

\begin{lemma} For $\omega$ defined in Eq(\ref{HJB.defomegaPC}), $\theta_\eta(\omega)=0$.
\end{lemma}

\begin{proof}

\begin{equation}
\begin{split}
\theta_\eta (\omega)&= \theta_\eta(dS+Edt+B d q)\\
&= d(\theta_\eta(S))+ \theta_\eta(E)d(\theta_\eta(t))+\theta_\eta(B) d(\theta_\eta(q))\\
&=d(-\gamma \ln(\eta))+\gamma\dfrac{1}{\eta}\dfrac{\partial \eta}{\partial t}dt + \gamma \dfrac{1}{\eta} \dfrac{\partial \eta}{\partial q}dq\\
&= 0.
\end{split}
\end{equation}

When the potential $V=0$, it is contained Proposition 5.2 of \cite{LZ1}
\end{proof}

\begin{definition}
For $(N_1,N_2)\in \mathcal{H}_V^2$, we set 
\begin{equation}
\Omega_\eta (N_1,N_2) :=\theta_\eta(\Omega(N_1,N_2)) \in \mathcal{C^\infty}(\mathbb{R}^2)
\end{equation}
 (see \cite{LZ1}, \textit{Lemma 5.5} ; that form had already been defined probabilistically in \cite{Zambrini2001}).\\
\end{definition}

Furthermore, for $N \in \mathcal{H}_V$, we shall write (as in \cite{LZ2}, p.211) \\
\begin{equation}
\begin{split}
F_N:&= N\rfloor \omega\\
(&=\omega(N)).
\end{split}
\end{equation}

We have :
\begin{equation}
\begin{split}
F_N:&= N\rfloor (dS+Edt+Bdq)\\
&= N\rfloor dS+(N\rfloor E) dt + E(N\rfloor dt) +(N\rfloor B)dq +B(N\rfloor dq)\\
&= \mathcal{L}_N(S)+E \mathcal{L}_N(t)+B \mathcal{L}_N(q)\\
&=N^S+E N^t+BN^q
\end{split}
\end{equation}
(see for example [HE], p.654) will play the role of a deformed contact Hamiltonian. For the relations with the usual concept of Hamiltonian in classical mechanics, cf \S 2 in \cite{LZ2}.
\begin{proposition}
For each $N\in \mathcal{H}_V$ and each $\eta > 0$ solution of Eq(\ref{eqdep}), one has :
\begin{equation}
\widetilde{N}(\eta)=-\dfrac{1}{\gamma} \eta \theta_\eta(F_N).
\end{equation}
\label{Ntilde}
\end{proposition}

\begin{remark} \textit{Lemma 5.4 (i)} of \cite{LZ1} follows at once from the fact that $ \mathcal{V}_N= -\widetilde{N}$ (\cite{LZ2}, p.194).
\end{remark}

\begin{proof}
\begin{equation}
\begin{split}
\widetilde{N}(\eta)=&-N^t \dfrac{\partial \eta }{\partial t}- N^q \dfrac{\partial \eta}{\partial q}- \dfrac{1}{\gamma} N^S \eta\\
&= -\dfrac{1}{\gamma}\eta \left[N^t \dfrac{\gamma}{\eta} \dfrac{\partial \eta }{\partial t}+N^q \dfrac{\gamma}{\eta}\dfrac{\partial \eta}{\partial q}+N^S \right]
\end{split}
\end{equation}
Using the definition of $ \theta_\eta$ and the fact Eq(\ref{HJB.defISO}) that $ N^t$, $N^q$ and $N^S$ only depend upon $t$ and $q$, one finds that: 
\begin{equation}
\begin{split}
\widetilde{N}(\eta) =& - \dfrac{1}{\gamma} \eta \left[ \theta_\eta (N^t) \theta_ \eta(E) +\theta_\eta (N^q) \theta_\eta(B) + \theta_\eta(N^S) \right]\\
=& -\dfrac{1}{\gamma}\eta \theta_\eta (N^tE+N^q B+N^S)\\
=& -\dfrac{1}{\gamma}\eta \theta_\eta(F_N)
\end{split}
\end{equation}
\end{proof}

The following result is of a purely algebro-geometrical nature.\\
\begin{lemma}(see \textit{Definition 4.1 in \cite{LZ1}, in the case $V=0$})\\
$\forall (\delta, \delta') \in (T^1M)^2$ : \\
\begin{equation}
\Omega (\delta, \delta')= \left(\delta(B) \delta'(q) - \delta(q) \delta'(B)\right)+\left(\delta(E) \delta'(t)-\delta(t)\delta'(E) \right).
\end{equation}
\end{lemma}

\begin{proof}
\begin{equation}
\begin{split}
\Omega(\delta, \delta')=& d \omega (\delta, \delta')\\
& [\text{by definition of $\Omega$}]\\
=& \delta(\omega(\delta'))- \delta'(\omega(\delta))- \omega([\delta, \delta'])\\
&[\text{see e.g \cite{docarmo}, \textit{Proposition 2, p.49}}]\\
=& \delta \left(E \delta'(t) +B \delta'(q) + \delta'(S) \right) - \delta' \left(E \delta(t) +B \delta(q) + \delta(S) \right) \\
&- \left( E [\delta, \delta'] (t) + B[\delta, \delta'](q)+[\delta, \delta'](S)\right)\\
=& \delta(E) \delta'(t) +E \delta(\delta'(t)) + \delta(B)\delta'(q) +B \delta(\delta'(q))+\delta(\delta'(S))\\
&-( \delta'(E) \delta(t) +E \delta'(\delta(t)) + \delta'(B)\delta(q) +B \delta'(\delta(q))+\delta'(\delta(S)))\\
&-E \left(\delta (\delta'(t))-\delta'(\delta(t))\right) - B \left(\delta (\delta'(q))-\delta'(\delta(q))\right)- \left(\delta (\delta'(S))-\delta'(\delta(S))\right)\\
=& \left(\delta(E) \delta'(t)- \delta'(E)\delta(t)\right)
+\left(\delta(B) \delta'(q)- \delta'(B)\delta(q)\right).
\end{split}
\end{equation}
\end{proof}

\begin{theorem}
For each $ N \in \mathcal{H}_V$ and each $\delta \in T^1M$, one has : \\
\begin{equation}
\Omega(N, \delta)= - \delta (F_N).
\label{propOmega}
\end{equation}
\end{theorem}

\begin{remark} In case $V= 0$, we recover \textit{Proposition 4.2} in \cite{LZ1}.
\end{remark}

\begin{proof}
One has : \\
\begin{equation}
\begin{split}
\Omega(N, \delta)=& d\omega(N, \delta)\\
=& N(\omega(\delta))- \delta(\omega(N))-\omega([N, \delta]),
\end{split}
\end{equation}
according to the same formula as above.\\
But $\mathcal{L}_N(\omega)=0$ (see \cite{LZ2}, p.214), whence 
\begin{equation}
\begin{split}
N(\omega(\delta))=& N(\delta \rfloor \omega)\\
=& \mathcal{L}_N(\delta \rfloor \omega)\\
=&[N,\delta]\rfloor \omega\\
=& \omega([N, \delta]),\\
\end{split}
\end{equation}
whence,
\begin{equation}
\begin{split}
\Omega(N, \delta) =& \omega ([N, \delta]) - \delta(\omega(N))- \omega([N,\delta])\\
=& - \delta (\omega(N))
=- \delta(F_N).
\end{split}
\end{equation}
\end{proof}

\begin{corollary}
\label{lemma7}
\label{defOmega}
For all $(N, N') \in \mathcal{H}_V^2$ : \\
\begin{equation}
\Omega_\eta(N, N')=- \gamma \dfrac{\widetilde{[N, N']}(\eta)}{\eta}.
\end{equation}
\end{corollary}

\begin{remark} In case $V=0$, one recovers \cite{LZ1}, \textit{Theorem (5.6)}.
\end{remark}

\begin{proof}
\begin{equation}
\begin{split}
\Omega(N, N')=&d\omega (N, N')\\
=&N(\omega (N'))-N'(\omega(N))-\omega([N, N']),
\end{split}
\end{equation}
once again via the same formula.\\
But 
\begin{equation}
\begin{split}
N(\omega (N'))=& \mathcal{L}_N(\omega(N'))\\
=&\mathcal{L}_N(N' \rfloor \omega)\\
=& [N,N']\rfloor \omega + N' \rfloor \mathcal{L}_N(\omega)\\
&[\text{again, see \cite{HE71} p.54}]\\
=& [N, N']\rfloor \omega\\
=& \omega([N, N']), 
\end{split}
\end{equation}
and similarly \\
\begin{equation}
N'(\omega(N))=\omega([N', N]),
\end{equation}
whence 
\begin{equation}
\begin{split}
\Omega (N, N')=& \omega([N, N'])- \omega([N', N])- \omega([N, N'])\\
=& -\omega ([N', N])\\
=& \omega ([N, N'])\\
=& F_{[N,N']}
\end{split}
\end{equation}
and
\begin{equation}
\begin{split}
\Omega_\eta(N,N')=&\theta_\eta(\Omega(N,N'))\\
=& \theta_\eta(F_{[N,N']})\\
=&-\dfrac{\gamma}{\eta}\widetilde{[N, N']}(\eta)\\
&[\text{by \textit{Proposition \ref{Ntilde}} applied to }[N,N']]\\
=& - \gamma \dfrac{\widetilde{[N,N']}(\eta)}{\eta}.
\end{split}
\end{equation}
\end{proof}

\begin{proposition}
For $(N, N') \in \mathcal{J}_V^2$\\
\begin{equation}
\Omega_\eta(N,N') = \gamma \dfrac{[N, N'](\eta)}{\eta}.
\end{equation}
\end{proposition}

\begin{remark}
In case $V=0$, we recover \textit{Corollary} 5.7 in \cite{LZ1} (where $ \mathcal{J}_0$ was denoted by $ \mathcal{H}_2$).
\end{remark}

\begin{proof}
As $N\in \mathcal{J}_V$ and $N' \in \mathcal{J}_V$, one has $[N,N']^S=0$ (see the proof of \textit{Lemma \ref{Jv-Kv})}, whence
\begin{equation}
\begin{split}
-\widetilde{[N, N']}=&[N,N']^t \dfrac{\partial}{ \partial t}+ [N, N']^q \dfrac{\partial}{\partial q}+ \dfrac{1}{\gamma}[N, N']^S \cdot\\
=& [N,N']^t\dfrac{\partial}{\partial t}+[N,N']^q\dfrac{\partial}{\partial q}
\end{split}
\end{equation}
and
\begin{equation}
\begin{split}
- \widetilde{[N, N']}(\eta) =& [N, N']^t \dfrac{\partial \eta}{\partial t}+ [N, N']^q \dfrac{\partial \eta}{\partial q}\\
=& [N, N'](\eta), 
\end{split}
\end{equation}
as $\eta$ is a function of $(q,t)$.\\
Therefore \\
\begin{equation}
\begin{split}
\Omega(N,N')_\eta=& -\gamma \dfrac{\widetilde{[N, N' ]}(\eta)}{\eta} \qquad \qquad [\text{by \textit{Corollary \ref{lemma7}} }]\\
=& \gamma \dfrac{[N, N'](\eta)}{\eta}
\end{split}
\end{equation}
as desired. 
\end{proof}

\section{Construction of martingales}
In the stochastic deformation of classical mechanics advocated here the crucial role of first integral or constant of motion is played by martingales of underlying diffusions. This aspect has already been analyzed in the context of a stochastic Noether Theorem \cite{TZ}. We are going to show that our stochastic $2$-form $\Omega$. is also a martingale.

Let :
\begin{equation}
\begin{split}
D_H=& \dfrac{\partial}{\partial t} - \dfrac{1}{\gamma}H\\
=&\dfrac{\partial}{\partial t}+ \dfrac{\gamma}{2} \dfrac{\partial^2}{\partial q}- \dfrac{1}{\gamma}V \cdot
\end{split}
\end{equation}

\begin{lemma}
For each $N\in \mathcal{H}_V$, $[\widetilde{N}, D_H]= -T_N'D_H$
\label{lemmesol}
\end{lemma}

\begin{proof}
The commutator splits into nine terms :
\begin{equation}
\begin{split}
-[\widetilde{N},D_H]=[-\widetilde{N}, D_H]=& [N^t \dfrac{\partial}{\partial t} +N^q \dfrac{\partial}{\partial q}+ \dfrac{1}{\gamma}N^S \cdot, \dfrac{\partial }{\partial t }+ \dfrac{\gamma}{2}\dfrac{\partial^2}{\partial q^2}-\dfrac{1}{\gamma}V \cdot ]\\
=& - \dfrac{\partial N^t}{\partial t} \dfrac{\partial }{\partial t}- \dfrac{\gamma}{2}( \dfrac{\partial N^t}{ \partial q} \dfrac{\partial^2}{ \partial q \partial t}+\dfrac{\partial^2}{\partial q^2} \dfrac{\partial}{\partial t}) \\
&- \dfrac{1}{\gamma}N^t \dfrac{ \partial V}{\partial t} \cdot  \dfrac{\partial N^q}{\partial t} \dfrac{\partial}{\partial q} -\dfrac{\gamma}{2}( \dfrac{\partial^2 N^q}{\partial q}\dfrac{\partial}{\partial q}+2 \dfrac{\partial N^q}{\partial q} \dfrac{\partial^2}{\partial q^2})\\
&- \dfrac{1}{\gamma}N^q \dfrac{\partial V}{\partial q} \cdot - \dfrac{1}{\gamma}\dfrac{\partial N^S}{\partial t}\cdot
 - \dfrac{1}{2} ( \dfrac{\partial^2 N^S}{\partial q^2}\cdot+2 \dfrac{\partial N^S}{\partial q}\dfrac{\partial}{\partial q}) +0.
 \end{split}
\end{equation}
As $N^t$ depends only upon $t$ (see Eq(\ref{HJB.defISO})), this reduces to : 
\begin{equation}
\begin{split}
&- \dfrac{\partial N^t}{\partial t} \dfrac{\partial }{\partial t} + \left( -\dfrac{\partial N^q}{\partial t }- \dfrac{\gamma}{2} \dfrac{\partial^2 N^q}{\partial q^2}- \dfrac{\partial N^S}{\partial q} \right) \dfrac{\partial}{\partial q}
- \gamma \dfrac{\partial N^q}{\partial q} \dfrac{\partial^2}{\partial q^2}\\
& +\left(- \dfrac{1}{\gamma}N^t \dfrac{\partial V}{\partial t}- \dfrac{1}{\gamma}N^q \dfrac{\partial V}{\partial q}- \dfrac{1}{\gamma} \dfrac{\partial N^S}{\partial q}- \dfrac{1}{2}\dfrac{\partial^2 N^S}{\partial q^2}\right) \cdot
\end{split}
\end{equation}
From the formulas recalled in \S \ref{HJB}, it appears that
\begin{equation}
\begin{split}
\dfrac{\partial N^t}{\partial t}=& \dfrac{\partial T_N}{\partial t}= T_N'(t),\\
\dfrac{\partial N^q}{\partial t} + \dfrac{\gamma}{2}\dfrac{\partial^2 N^q}{\partial q^2}+\dfrac{\partial N^S}{\partial q}=& \dfrac{1}{2}T''_N(t)q+l'(t)+\dfrac{\gamma}{2}\cdot 0 + \left(-\dfrac{1}{2} T''_N(t)q-l'(t) \right)\\
=&0, \\
\dfrac{\partial N^q}{\partial q}=& \dfrac{1}{2}T_N'(t), \text{ and }\\
-\dfrac{1}{\gamma}N  \dfrac{\partial V}{\partial t}- \dfrac{1}{\gamma}N^q \dfrac{\partial V}{\partial q}
=&-\dfrac{1}{\gamma} \dfrac{\partial N^S}{\partial t}-\dfrac{1}{2}\dfrac{\partial^2 N^S}{\partial q^2}\\
=& -\dfrac{1}{\gamma}T_N \dfrac{\partial V}{\partial t}- \dfrac{1}{\gamma}(\dfrac{1}{2}T_N'(t) q +l(t)) \dfrac{\partial V}{\partial q}+ \dfrac{1}{\gamma} \dfrac{\partial \phi}{\partial t}+\dfrac{1}{2}\dfrac{\partial^2 \phi}{\partial q^2}\\
=& - \dfrac{1}{\gamma} \left(- \dfrac{\partial \phi}{\partial t}+ T_N \dfrac{\partial V}{\partial t}+ (\dfrac{1}{2}T_N' +l) \dfrac{\partial V}{\partial q} - \dfrac{\gamma}{2} \dfrac{\partial^2 \phi }{\partial q^2}+T'_N(t) V \right) \\
&+\dfrac{1}{\gamma}T_N'(t) V\\
=& \dfrac{1}{\gamma}T_N'(t) V , \text{ whence }\\
[\widetilde{N}, D_H] =& -T_N'(t) \dfrac{\partial}{\partial t}-\dfrac{\gamma}{2} T_N'(t) \dfrac{\partial^2}{\partial q^2} + \dfrac{1}{\gamma} T_N'(t) V \cdot\\
=& -T_N'(t) (\dfrac{\partial }{\partial t}+ \dfrac{\gamma}{2} \dfrac{\partial^2}{\partial q^2}- \dfrac{1}{\gamma}V \cdot)\\
=&-T_N'(t)D_H.
\end{split}
\end{equation}

\end{proof}

The claim of Lemma \ref{lemmesol} in fact implies that $\widetilde{N}$ belong to the symmetry algebra of Eq(\ref{HJB.edpeta}). Indeed, 

\begin{corollary}
Let $\eta$ be a solution of Eq(\ref{eqdep}), then, for each $N\in \mathcal{H}_V$, $\widetilde{N}(\eta)$ is also a solution of Eq(\ref{eqdep}).
\label{ntildesol}
\end{corollary}

\begin{proof}
A function $\eta (t,q)$ satisfies Eq(\ref{eqdep}) if and only if $D_H(\eta)=0$, but this implies
\begin{equation}
\begin{split}
D_H(\widetilde{N}(\eta)) =& (D_H \widetilde{N})(\eta)\\
=& \left([D_H, \widetilde{N}] + \widetilde{N} D_H \right)(\eta)\\
=&\left(T'_N(t)D_H+ \widetilde{N}D_H \right)(\eta) \qquad
[\text{by \textit{Lemma \ref{lemmesol}}} ]\\
=& T'_N(t) D_H(\eta)+ \widetilde{N}(D_H(\eta))\\
=& 0. \\
\end{split}
\end{equation}
\end{proof}

The main result of this section is the 

\begin{theorem}
$(\Omega_\eta(N,M))(t,z(t))$, where $z(\cdot)$ solves Eq(\ref{Intro.defz}) is a martingale with respect to the increasing filtration $\mathcal{P}_t$ associated with this stochastic differential equation. \label{martingales}
\end{theorem}

\begin{proof}
The generator of the process can be written :
\begin{equation}
D_t= \dfrac{\partial }{\partial t}+ \widetilde{B}\dfrac{\partial}{\partial q}+\dfrac{\gamma}{2}\dfrac{\partial^2}{\partial q^2}
\end{equation}
($\widetilde{D}$ in \cite{LZ1}, \cite{LZ2}).
It's simple to see that (as $\eta$ is a solution of Eq(\ref{eqdep})) :
\begin{equation}
D_tf(t,q)= \dfrac{1}{\eta}(\dfrac{\partial}{\partial t}-\dfrac{1}{\gamma}H)(f \eta),
\end{equation}
In particular, using \textit{Corollary \ref{defOmega}}, one finds  : \\
\begin{equation}
D_t(\Omega_\eta(N,M))=\dfrac{1}{\eta}(\dfrac{\partial}{\partial t}-\dfrac{1}{\gamma}H)(-\gamma \dfrac{\widetilde{[N,M]}(\eta)}{\eta} \eta).
\end{equation}
As $N$ and $M$ belong to $\mathcal{H}_V$, so, according to \textit{Corollary \ref{ntildesol}} applied to $[N,M]$, $\widetilde{[N,M]}(\eta)$ is also a solution of Eq(\ref{eqdep}). We then  have :
\begin{equation}
D_t(\Omega_\eta(N,M))=0.
\end{equation}
\end{proof}

\section{Parametrization of a one--factor affine model}
\label{aff1}

As general references concerning affine models, we shall use H\'enon's PhD thesis (\cite{Henon}) as well as Leblanc and Scaillet's seminal paper (\cite{LeblancScaillet}).

An \it one--factor affine interest rate model \rm  is characterized by the instantaneous rate $r(t)$, satisfying
the following stochastic differential equation :
\begin{eqnarray}        
dr(t)=\sqrt{\alpha r(t) + \beta} \,\, dw(t) +(\phi - \lambda r(t)) \,\, dt 
\label{parametrisation1}
\end{eqnarray}
under the risk--neutral probability $Q$
($\alpha=0$ corresponds to the so--called Vasicek model, and $\beta=0$ corresponds to the
Cox-Ingersoll-Ross model ; cf. \cite{LeblancScaillet}).

Assuming $\alpha>0$, let us set $$\widetilde{\phi}=_{def}\phi + \displaystyle\frac{\lambda \beta}{\alpha}\,\, ,$$  $$\delta=_{def}\displaystyle\frac{4\widetilde{\phi}}{\alpha}\,\, ,$$ and let us also assume
that $\widetilde{\phi}\geq 0$.

The following two quantities will play an important role :

\begin{eqnarray}
C 
&:=&\frac{\alpha^{2}}{8}(\widetilde{\phi}-\frac{\alpha}{4})(\widetilde{\phi}-\frac{3\alpha}{4}) \nonumber \\
&=&\frac{\alpha^{4}}{128}(\delta-1)(\delta-3) \nonumber
\end{eqnarray}

and

$$
D:=\frac{\lambda^{2}}{8}\,\, .
$$

Let us set $X_{t}=\alpha r(t)+\beta$.
\begin{proposition}
\label{StrongSolu}Let $r_{0}\in \mathbf R$ ;
then the stochastic differential equation
\begin{eqnarray}
dr(t)=\sqrt{\vert \alpha r(t) + \beta \vert} \,\, dw(t) +(\phi - \lambda r(t)) \,\, dt 
\label{parametrisation2}
\end{eqnarray}
has a unique strong solution such that $r(0)=r_{0}$.
Furthermore, in case that $$\alpha r_{0}+\beta \geq 0\,\, ,$$ 
 one has  $\alpha r(t)+\beta \geq 0$ for all $t\geq 0$ ;
in particular, $r(t)$ satisfies Eq(\ref{parametrisation1}).
\end{proposition}
\begin{proof}

It is easy to see that,
in terms of $X_{t}$, Eq(\ref{parametrisation2}) becomes :

\begin{eqnarray}
dX_{t} 
&=&\alpha dr(t) \nonumber \\
&=&\alpha(\sqrt{\vert X_{t} \vert}dw(t)+(\phi-\lambda\frac{X_{t}-\beta}{\alpha})dt) \nonumber \\
&=&\alpha\sqrt{\vert X_{t} \vert}dw(t)+(\alpha\widetilde{\phi}-\lambda X_{t})dt  \,\, . \nonumber
\end{eqnarray}

We are therefore in the situation of $(1)$, p.313, in \cite{GoingYor}, with $c=\alpha$,
$a=\alpha\widetilde{\phi}$, and $b=-\lambda$ ; the result follows.

In case $\lambda\neq 0$, one may also refer to \cite{Henon}, p.55, Proposition 12.1, with $\sigma=\alpha$, $\kappa=\lambda$ and
$a=\displaystyle\frac{\alpha\widetilde{\phi}}{\lambda}$.
\end{proof}

We shall henceforth assume all of the hypotheses of Proposition \ref{StrongSolu} to be satisfied.
\begin{corollary}
\label{BESQ} One has
$$
\left\{
\begin{array}{lr}
X_{t}=e^{-\lambda t} \, Y(\frac{\alpha^{2}(e^{\lambda t}-1)}{4\lambda}) \,\, \text{for} \,\, \lambda \neq 0 \nonumber \,\, , \text{and} \\
X_{t}=Y(\frac{\alpha^{2}t}{4}) \,\, \text{for} \,\, \lambda = 0 \nonumber \\
\end{array}
\right.
$$
where $Y$ is a $BESQ^{\delta}$(squared Bessel process with parameter 
$\delta$) having initial value $Y_{0}=\alpha r_{0}+\beta$.
\end{corollary}
\begin{proof}In case $\lambda\neq 0$, one applies the result of \cite{GoingYor}, p. 314.
For $\lambda=0$, let
$$
Z_{t}:=\displaystyle\frac{4}{\alpha^{2}}X_{t}\,\, ;
$$
it appears that 
$$
dZ_{t}=2\sqrt{\vert Z_{t} \vert}dw(t)+\delta dt \,\, ,
$$
whence $Z_{t}$ is a $BESQ^{\delta}$--process.
As
$$
X_{t}:=\displaystyle\frac{\alpha^{2}}{4}Z_{t}\,\, ,
$$
the scaling property of Bessel processes 
yields the result.
\end{proof}

\begin{theorem}\label{theo1Aff1}If $\delta\geq 2$ one has, almost surely :
$$
\forall t>0  \,\,\, X_{t}>0\,\, ;
$$
on the other hand, if $\delta<2$, almost surely there is a $t>0$ such that $X_{t}=0$.
\end{theorem}
\begin{proof}We apply Corollary 1, p. 317, from \cite{LeblancScaillet}(12.2)  yielding that 
\begin{eqnarray}
\forall t>0, \,\,\, X_{t}>0 \,\, 
\end{eqnarray}
if and only if $\delta\geq 2$.

One may also use \cite{Henon}, p.56, from which follows that $(3.3)$
is equivalent to 
$$\displaystyle\frac{2a\kappa}{\sigma^{2}}\geq 1 \,\, ; $$
 but,
according to the above identifications,
$$
\displaystyle\frac{2a\kappa}{\sigma^{2}}=\displaystyle\frac{2{\displaystyle\frac{\alpha\widetilde{\phi}}{\lambda}}{\lambda}}{\alpha^{2}}
=\displaystyle\frac{2\widetilde{\phi}}{\alpha}=\displaystyle\frac{\delta}{2}\,\, .
$$
\end{proof}

Our main result is the following :
\begin{theorem}
\label{bernstein}
 Let us define the process
$$
z(t)=\sqrt{X_{t}}
$$
and the stopping time
$$
T=\inf\{t>0 \vert X_{t}=0\} ;
$$
as seen in Theorem \ref{theo1Aff1}, $T=+\infty \,\, 
a.s.$ for $\delta\geq 2$,
and 
$T<+\infty \,\, a.s.$ for $\delta<2$.
\, 
Then there exists a Bernstein process $y(t)$
for 
$$
\gamma=\frac{\alpha^2}{4}\,\,
$$
and the potential
$$
V(t,q)=\frac{C}{q^{2}}+Dq^{2}\,\, .
$$
such that
$$
\forall t\in [0,T[,\,\, z(t)=y(t) \,\, .
$$
In particular, for $\delta\geq 2$, $z$ itself is a Bernstein process.
\end{theorem}

\begin{proof}
One has (cf. Proposition (\ref{StrongSolu}) and its proof)

\begin{eqnarray}
dX_{t} \nonumber
&=&\alpha\sqrt{X_{t}}dw(t)+(\alpha\widetilde{\phi}-\lambda X_{t})dt \nonumber \\
&=&\alpha z(t)dw(t)+(\alpha\widetilde{\phi}-\lambda z(t)^{2})dt \,\, .\nonumber
\end{eqnarray}

Taking now  $f(x)=\sqrt{x}$, we have $$\forall x>0\,\, f^{'}(x)=\frac{1}{2\sqrt{x}}\,\,\text{and}
\,\,f^{''}(x)=-\frac{1}{4}x^{-\frac{3}{2}}\,\, ,$$ therefore, for all $t\in ]0,T[$,
$f^{'}(X_{t})=\displaystyle\frac{1}{2z(t)}$
and $f^{''}(X_{t})=-\frac{1}{4}z(t)^{-3}$.
The application of It\^o's formula now gives :

\begin{eqnarray}
dz(t) \nonumber
&=&d(f(X_{t})) \nonumber \\
&=&f^{'}(X_{t})dX_{t}+\frac{1}{2}f^{''}(X_{t})(dX_{t})^{2} \nonumber \\
&=&\frac{1}{2z(t)}(\alpha z(t)dw(t)+(\alpha\widetilde{\phi}-\lambda z(t)^{2})dt)
-\frac{1}{8}z(t)^{-3}\alpha^{2}z(t)^{2}dt \nonumber \\
&=&\displaystyle\frac{\alpha}{2}dw(t)+\frac{1}{8z(t)}(4\alpha\widetilde{\phi}-4\lambda z(t)^{2}-\alpha^{2})dt \,\, .\nonumber 
\end{eqnarray}

Let us now define $\eta$ by
\begin{eqnarray}
\eta(t,q) 
&:=&\exp({\displaystyle\frac{\lambda\widetilde{\phi}t}{\alpha}-\displaystyle\frac{\lambda q^{2}}{\alpha^{2}}})
q^{\displaystyle\frac{2\widetilde{\phi}}{\alpha}-\displaystyle\frac{1}{2}}
\,\, \nonumber \\
&=&\exp({\displaystyle\frac{\lambda \delta t}{4}-\displaystyle\frac{\lambda q^{2}}{\alpha^{2}}})
q^{\displaystyle\frac{\delta - 1}{2}} \nonumber \,\, .
\end{eqnarray}

It is easy to check that $\eta$ solves the equation
$$
\gamma\displaystyle\frac{\partial \eta}{\partial t}=-\displaystyle\frac{\gamma^2}{2}
\displaystyle\frac{\partial^{2}\eta}{\partial q^{2}}+V\eta  
$$
for
$$
V=\frac{C}{q^{2}}+Dq^{2}\,\, ;
$$
in other words,
\begin{eqnarray}
S \nonumber
&:=&-\gamma\ln(\eta) \nonumber \\
&=&-\gamma(\displaystyle\frac{\lambda \delta t}{4}-\displaystyle\frac{\lambda q^{2}}{\alpha^{2}}
+\displaystyle\frac{\delta - 1}{2}\ln(q)) \nonumber \,\, \\
&=&-\displaystyle\frac{\alpha^{2}\lambda \delta t}{16}+\displaystyle\frac{\lambda q^{2}}{4}
-\alpha^{2}(\frac{\delta - 1}{8})\ln(q) \nonumber \,\, 
\end{eqnarray}
satisfies $(\mathcal H\mathcal J\mathcal B^{V})$.
Furthermore we 
have :

\begin{eqnarray}
\widetilde{B} 
&:=&\gamma\displaystyle\frac{\displaystyle\frac{\partial \eta}{\partial q}}{\eta} \nonumber \\
&=&-\displaystyle\frac{\partial S}{\partial q} \nonumber \\ 
&=&-\displaystyle\frac{\lambda q}{2}+\displaystyle\frac{\alpha^{2}(\delta - 1)}{8q} \nonumber \\
&=& \displaystyle\frac{1}{8q}(\alpha^{2}\delta-\alpha^{2}-4\lambda q^{2}) \nonumber
\end{eqnarray}
whence
$$
\widetilde{B}(t,z(t))=\displaystyle\frac{1}{8z(t)}(4\alpha\widetilde{\phi}-\alpha^{2}-4\lambda z(t)^{2})
$$
and $z$ satisfies the stochastic differential equation associated with $\eta$ :
$$
\forall t\in ]0,T[ \,\,\, dz(t)=\dfrac{\alpha}{2} dw(t)+\widetilde{B}(t,z(t))dt
$$
(as in \cite{LZ1}, \S 1, equation ($\mathcal{B}_1)$) ; the result follows.
\end{proof}

\section{Isovectors in the case $V(t,q)=\frac{C}{q^2}+Dq^2$}
\label{Isovectors}
In this section, we are going to consider a class of potentials $V=\frac{C}{q^2}+Dq^2$  for various $C$ and $D$. For $D=0$, $C\neq 0$ this is the case of a radial free quantum particule. When $C \neq 0$, and $D>0$, this is the potential of a radial harmonic oscillator. When $C \neq 0$, and $D<0$, we have a radial repulsive oscillator. Those cases have been studied in the context of symmetries for Schrödinger (and Heat) equations, in particular by Willard Miller Jr.. The methods used here are different and, of course, our main priority is probabilistic : how to solve some stochastic differential equations.

\textit{Theorem (\ref{bernstein})} motivates us to find the explicit form of the isovectors for the potential :
\begin{equation}
V(t,q)=\frac{C}{q^2}+Dq^2.
\end{equation}\\
 Let us call $\mathcal{H}_{(C,D)}=\mathcal{H}_V$ and $\widetilde{\mathcal{H}}_{(C,D)}=\widetilde{\mathcal{H}}_V$.
 
We have to solve :
\begin{equation}
\begin{split}
& q^2[-\frac{1}{4}T_N'''+2D T'_N]+q[-l''+2Dl]+\sigma'-\frac{\gamma}{4}T_N'''-\frac{2Cl}{q^3}=0.
\label{eqiso}
\end{split}
\end{equation}
 
Using the fact that $T_N$, $l$ and $\sigma$ are independent from $q$, Eq(\ref{eqiso}) is equivalent to :

 \begin{equation}
 \begin{cases}
2Cl =0 \\
l''=2Dl\\
\sigma'=\frac{\gamma}{4}T_N''\\
T_N'''=8DT_N'.
\end{cases}
\label{sys}
\end{equation}

\begin{enumerate}[1)]
\item $C \neq 0$. \\
Eq(\ref{sys}) gives $l=0$, and becomes
 \begin{equation}\begin{cases}
\sigma'=\frac{\gamma}{4}T_N''\\
T_N'''=8DT_N'.
\end{cases}\end{equation}

\begin{enumerate}[a)]
\item $D>0$.\\
We define $\varepsilon=\sqrt{8D}$ and find :
\begin{equation}
\begin{split}
T_N'&=C_1 e^{\varepsilon t}+C_2 e^{-\varepsilon t},\\
T_N&=\frac{C_1}{\varepsilon} e^{\varepsilon t}-\frac{C_2}{\varepsilon} e^{-\varepsilon t}+C_3,\\
T_N''&=\varepsilon(C_1 e^{\varepsilon t}-C_2 e^{-\varepsilon t}),
\end{split}
\end{equation}

\begin{equation}
\begin{split}
\sigma&=\frac{\gamma}{4}T_N'+C_4\\
&=\frac{\gamma C_1}{4} e^{\varepsilon t}+\frac{\gamma C_2}{4} e^{-\varepsilon t}+C_4.
\end{split}
\end{equation}

By replacing in Eq(\ref{HJB.defISO}) (and considering that $N^S$ is independent from $S$) :

\begin{equation}
\begin{split}
&N^q=\frac{1}{2}q(C_1 e^{\varepsilon t}+C_2 e^{-\varepsilon t}),\\
&N^t=\frac{C_1}{\varepsilon} e^{\varepsilon t}-\frac{C_2}{\varepsilon} e^{-\varepsilon t}+C_3,\\
&N^S=-\phi. \label{mira2}
\end{split}
\end{equation}

Using :

\begin{equation}
\begin{split}
\phi(t,q)&=\frac{1}{4}T_N''q^2+l'q-\sigma(t)\\
&=\frac{1}{4}(\varepsilon C_1 e^{\varepsilon t}-\varepsilon C_2 e^{-\varepsilon t})q^2-\frac{\gamma C_1}{4} e^{\varepsilon t}-\frac{\gamma C_2}{4} e^{-\varepsilon t}-C_4,
\end{split}
\end{equation}

we find :

\begin{equation}
\begin{split}
\frac{\partial \phi}{\partial q}(t,q)&=\frac{\varepsilon}{2}( C_1 e^{\varepsilon t}-C_2 e^{-\varepsilon t})q,\\
\frac{\partial \phi}{\partial t}(t,q)&=\frac{\varepsilon^2 q^2-\varepsilon \gamma}{4} C_1 e^{\varepsilon t}+\frac{\varepsilon^2 q^2+\varepsilon \gamma}{4} C_2 e^{-\varepsilon t}.
\end{split}
\end{equation}

That gives
\begin{equation}
\begin{split}
&N^q=\frac{1}{2}q(C_1 e^{\varepsilon t}+C_2 e^{-\varepsilon t}),\\
&N^t=\frac{C_1}{\varepsilon} e^{\varepsilon t}-\frac{C_2}{\varepsilon} e^{-\varepsilon t}+C_3,\\
&N^S=-\frac{1}{4}(\varepsilon C_1 e^{\varepsilon t}-\varepsilon C_2 e^{-\varepsilon t})q^2+\frac{\gamma C_1}{4} e^{\varepsilon t}+\frac{\gamma C_2}{4} e^{-\varepsilon t}+C_4.
\end{split}
\end{equation}

Let's define :
\begin{equation}
\widetilde{P}_i=-N^q \frac{\partial}{\partial q}-N^t \frac{\partial}{\partial t}-\dfrac{1}{\gamma} N^S ,
\end{equation}
for $C_i=1$ and $C_j=0$ if $j\neq i$.

We then obtain :

\begin{equation}
\begin{split}
\widetilde{P}_1&=-\frac{1}{2}q e^{\varepsilon t} \frac{\partial}{\partial q}-\frac{1}{\varepsilon} e^{\varepsilon t} \frac{\partial}{\partial t}-\frac{e^{\varepsilon t}}{4\gamma} (-q^2 \varepsilon+\gamma),\\
\widetilde{P}_2&=-\frac{1}{2}q e^{-\varepsilon t} \frac{\partial}{\partial q}+\frac{1}{\varepsilon} e^{-\varepsilon t} \frac{\partial}{\partial t}-\frac{e^{-\varepsilon t}}{4 \gamma}(q^2 \varepsilon+\gamma),  \\
\widetilde{P}_3&=-\frac{\partial}{\partial t},\\
\widetilde{P}_4&=-\frac{1}{\gamma}.
\end{split}
\end{equation}

\item $D=0$.

Using $T_N'''=0$ in Eq(\ref{eqiso}) :
\begin{equation}
\begin{split}
T_N&=C_1t^2+C_2t+C_3,\\
T_N'&=2C_1t+C_2,\\
T_N''&=2C_1.\\
\end{split}
\end{equation}

That gives :

\begin{equation}
\begin{split}
\sigma&=\frac{\gamma}{4}T_N'+\frac{\gamma}{4}C_2+C_4'\\
&=\frac{\gamma C_1}{2}t+\frac{\gamma}{4}C_2+C_4'.
\end{split}
\end{equation}
Let's define $C_4=\frac{\gamma}{4}C_2+C_4'$.

Using :
\begin{equation}
\begin{split}
\phi(t,q)&=\frac{1}{4}T_N''q^2+l'q-\sigma(t)\\
&=\frac{1}{2}C_1q^2-\frac{\gamma C_1t}{2}-C_4,
\end{split}
\end{equation}

we find :
\begin{equation}
\begin{split}
\frac{\partial \phi}{\partial q}(t,q)&=C_1 q,\\
\frac{\partial \phi}{\partial t}(t,q)&=-\frac{\gamma C_1}{2}.
\end{split}
\end{equation}

That gives :

\begin{equation}
\begin{split}
&N^q=C_1qt+\frac{C_2q}{2},\\
&N^t=C_1t^2+C_2t+C_3,\\
&N^S=\frac{C_1}{2}(\gamma t-q^2)+C_4.
\end{split}
\end{equation}

We obtain :

\begin{equation}
\begin{split}
\widetilde{P}_1&=-tq \frac{\partial}{\partial q}-t^2\frac{\partial}{\partial t}-\frac{\gamma t-q^2}{2\gamma},\\
\widetilde{P}_2&=-\frac{q}{2}\frac{\partial}{\partial q}-t \frac{\partial}{\partial t},\\
\widetilde{P}_3&=-\frac{\partial}{\partial t},\\
\widetilde{P}_4&=-\frac{1}{\gamma}.
\end{split}
\end{equation}

In the following those will be named $\widetilde{M}_1$, $\widetilde{M}_2$, $\widetilde{M}_3$ et $\widetilde{M}_4$.

\item $D<0$.

Let's define $\varepsilon=\sqrt{-8D}$

Then we have :

\begin{equation}
\begin{split}
T_N&=\frac{1}{\varepsilon}C_1 \sin (\varepsilon t)-\frac{1}{\varepsilon}C_2 \cos (\varepsilon t)+C_3,\\
T_N'&=C_1 \cos(\varepsilon t) +C_2 \sin (\varepsilon t),\\
T_N''&=-\varepsilon C_1 \sin(\varepsilon t) + \varepsilon C_2 \cos (\varepsilon t),\\
\sigma &= \frac{\gamma}{4}C_1 \cos (\varepsilon t)+\frac{\gamma}{4}C_2 \sin (\varepsilon t)+C_4.
\end{split}
\end{equation}

That gives :

\begin{equation}
\begin{split}
N^t &=\frac{1}{\varepsilon}C_1 \sin (\varepsilon t)-\frac{1}{\varepsilon}C_2 \cos (\varepsilon t)+C_3,\\
N^q &=\frac{q}{2}(C_1 \cos(\varepsilon t) +C_2 \sin (\varepsilon t)),\\
N^S &=-\frac{q^2}{4}\varepsilon (-C_1 \sin(\varepsilon t) +C_2 \cos (\varepsilon t)) + \frac{\gamma }{4}(C_1 \cos (\varepsilon t) +C_2 \sin (\varepsilon t))+C_4.
\end{split}
\end{equation}

Then we obtain :

\begin{equation}
\begin{split}
\widetilde{P}_1&=\frac{-q}{2}\cos(\varepsilon t)\frac{\partial}{\partial q}-\frac{1}{\varepsilon }\sin(\varepsilon t)\frac{\partial}{\partial t}-\frac{1}{\gamma}(\frac{1}{4}q^2\varepsilon \sin(\varepsilon t)+\frac{\gamma}{4}\cos(\varepsilon t)), \\
\widetilde{P}_2&=\frac{-q}{2}\sin(\varepsilon t)\frac{\partial}{\partial q}+\frac{1}{\varepsilon }\cos(\varepsilon t)\frac{\partial}{\partial t}-\frac{1}{\gamma}(-\frac{1}{4}q^2\varepsilon \cos(\varepsilon t)+\frac{\gamma}{4}\sin(\varepsilon t)), \\
\widetilde{P}_3&=-\frac{\partial}{\partial t},\\
\widetilde{P}_4&=-\frac{1}{\gamma}.
\end{split}
\end{equation}

\end{enumerate}

\item $C=0$.\\
Eq(\ref{sys}) becomes

\begin{equation}\begin{cases}
l'' =2Dl \\
\sigma'=\frac{\gamma}{4}T_N''\\
T_N'''=8DT_N'.
\end{cases}\end{equation}

\begin{enumerate}
\item $D>0$.\\
$T_N$ and its derivatives remain the same.\\
\begin{equation}
\begin{split}
l&=C_5e^{\frac{\varepsilon}{2} t}+C_6e^{-\frac{\varepsilon}{2} t},\\
l'&=\frac{\varepsilon}{2} C_5e^{\frac{\varepsilon}{2} t}-\frac{\varepsilon}{2} C_6e^{-\frac{\varepsilon}{2} t}.
\end{split}
\end{equation}

We obtain
\begin{equation}
\begin{split}
&N^q=\frac{1}{2}q(C_1 e^{\varepsilon t}+C_2 e^{-\varepsilon t})+C_5e^{\frac{\varepsilon}{2} t}+C_6e^{-\frac{\varepsilon}{2} t},\\
&N^t=\frac{C_1}{\varepsilon} e^{\varepsilon t}-\frac{C_2}{\varepsilon} e^{-\varepsilon t}+C_3,\\
&N^S=-\frac{1}{4}(\varepsilon C_1 e^{\varepsilon t}-\varepsilon C_2 e^{-\varepsilon t})q^2+\frac{\gamma C_1}{4} e^{\varepsilon t}+\frac{\gamma C_2}{4} e^{-\varepsilon t}+C_4-q\frac{\varepsilon}{2} C_5e^{\frac{\varepsilon}{2} t}+q\frac{\varepsilon}{2} C_6e^{-\frac{\varepsilon}{2} t}.
\end{split}
\end{equation}

The $\widetilde{P}_i$ remain the same for $i=1,2,3,4$ and we have :

\begin{equation}
\begin{split}
\widetilde{P}_5&=-e^{\frac{\varepsilon}{2} t} \frac{\partial}{\partial q}+\frac{1}{\gamma}q \frac{\varepsilon}{2} e^{\frac{\varepsilon}{2} t},\\
\widetilde{P}_6&=-e^{-\frac{\varepsilon}{2} t} \frac{\partial}{\partial q}-\frac{1}{\gamma}q \frac{\varepsilon}{2} e^{-\frac{\varepsilon}{2} t}.\\
\end{split}
\end{equation}

\item $D=0$.\\
$T_N$ and its derivatives remain the same.
\begin{equation}
\begin{split}
l &=C_5t+C_6,\\
l' &=C_5,
\end{split}
\end{equation}

\begin{equation}
\begin{split}
&N^q=C_1qt+\frac{C_2q}{2}+C_5t+C_6,\\
&N^t=C_1t^2+C_2t+C_3,\\
&N^S=\frac{C_1}{2}(\gamma t-q^2)+C_4-qC_5.
\end{split}
\end{equation}

The $\widetilde{P}_i$ remain the same for $i=1,2,3,4$ and we have :

\begin{equation}
\begin{split}
\widetilde{P}_5&=-t\frac{\partial}{\partial q}+\frac{1}{\gamma}q,\\
\widetilde{P}_6&=-\frac{\partial}{\partial q}.
\end{split}
\end{equation}

In the following those will be named $\widetilde{M}_5$ et $\widetilde{M}_6$.

\begin{remark}
\cite{LZ1} already gave those without the mapping from $N$ to $\widetilde{N}$ and had :
\begin{equation}
\begin{split}
&N_1=\frac{\partial}{\partial t}, \qquad N_2=\frac{\partial}{\partial q}, \qquad N_3=-\gamma \frac{\partial}{\partial S} \qquad N_4=2t \frac{\partial}{\partial t} +q \frac{\partial}{\partial q}- 2E \frac{\partial}{\partial E} -B \frac{\partial}{\partial B},\\
&N_5=-2t \frac{\partial}{\partial q} +2q \frac{\partial}{\partial S} +2B \frac{\partial}{\partial E} -2 \frac{\partial}{\partial B}, \\
&N_6=2t^2 \frac{\partial}{\partial t} +2qt \frac{\partial}{\partial q}+(\gamma t-q^2) \frac{\partial}{\partial S}-(2qB+4tE+\gamma)\frac{\partial}{\partial E}+2(q-tB)\frac{\partial}{\partial B}.
\end{split}
\end{equation}
That gives after mapping from $N$ to $\widetilde{N}$ :
\begin{equation}
\begin{split}
\label{ntilde}
&\widetilde{N}_1=-\frac{\partial}{\partial t}=\widetilde{M}_3, \qquad \widetilde{N}_2=-\frac{\partial}{\partial q}=\widetilde{M}_6, \qquad \widetilde{N}_3=1=-\gamma \widetilde{M}_4, \\
& \widetilde{N}_4=-2t\frac{\partial}{\partial t}-q\frac{\partial}{\partial q}=2\widetilde{M}_2,\\
&\widetilde{N}_5=2t \frac{\partial}{\partial q}-\frac{2}{\gamma}q=-2\widetilde{M}_5,\\
&\widetilde{N}_6=-2t^2 \frac{\partial}{\partial t}-2qt\frac{\partial}{\partial q}-\frac{(\gamma t-q^2)}{\gamma}=2\widetilde{M}_1.
\end{split}
\end{equation}
\end{remark}

\item $D<0$.\\
Let's define $\frac{\varepsilon}{2} =\sqrt{-2D}$.\\
$T_N$ and its derivatives remain the same.
\begin{equation}
\begin{split}
l &=C_5\cos(\frac{\varepsilon}{2} t)+C_6\sin (\frac{\varepsilon}{2} t),\\
l' &=-\frac{\varepsilon}{2} C_5\sin(\frac{\varepsilon}{2} t)+\frac{\varepsilon}{2} C_6\cos (\frac{\varepsilon}{2} t).
\end{split}
\end{equation}

\begin{equation}
\begin{split}
N^t &=\frac{1}{\varepsilon}C_1 \sin (\varepsilon t)-\frac{1}{\varepsilon}C_2 \cos (\varepsilon t)+C_3\\
N^q &=\frac{q}{2}(C_1 \cos(\varepsilon t) +C_2 \sin (\varepsilon t))+C_5\cos(\frac{\varepsilon}{2} t)+C_6\sin (\frac{\varepsilon}{2} t)\\
N^S &=-\frac{q^2}{4}\varepsilon (-C_1 \sin(\varepsilon t) +C_2 \cos (\varepsilon t)) + \frac{\gamma }{4}(C_1 \cos (\varepsilon t) +C_2 \sin (\varepsilon t)+C_4)\\
&\qquad +q\frac{\varepsilon}{2} C_5\sin(\frac{\varepsilon}{2} t)-q\frac{\varepsilon}{2} C_6\cos (\frac{\varepsilon}{2} t).
\end{split}
\end{equation}

The $\widetilde{P}_i$ remain the same for $i=1,2,3,4$ and we have :

\begin{equation}
\begin{split}
\widetilde{P}_5&=-\cos(\frac{\varepsilon}{2} t)\frac{\partial}{\partial q}-\frac{1}{\gamma}q\frac{\varepsilon}{2} \sin (\frac{\varepsilon}{2} t),\\
\widetilde{P}_6&=-\sin(\frac{\varepsilon}{2} t)\frac{\partial}{\partial q}+\frac{1}{\gamma}q\frac{\varepsilon}{2} \cos (\frac{\varepsilon}{2} t).\\
\end{split}
\end{equation}

\end{enumerate}
\end{enumerate}

We call ${\mathcal{P}}_{D}$ the space generated by $\{P_1,..., {P_4}\}$.\\
Then we have : 
\begin{theorem}
\begin{enumerate}
\item $\forall C \neq 0$, $\forall D$,  $\mathcal{H}_{(C,D)}=  {\mathcal{P}}_{D}$ ; in particular for $C\neq 0$, the algebra $\mathcal{H}_{(C,D)}$ does not depend on $C$.
\item $\dim {\mathcal{P}}_{D}=4$.
\item $\dim {\mathcal{H}}_{(0,D)}=6$.
\item ${\mathcal{P}}_{D}$ is a Lie sub algebra of $\mathcal{H}_{(0,D)}$.
\item $\forall C$, $\forall D$, $\mathcal{H}_{(C,D)} \subseteq \mathcal{H}_{(0,D)}  $ ( that had been proved in \cite{LZ2}, p.220 for $C>0$ et $D=0$).
\end{enumerate}
\end{theorem}
 
\begin{remark}The theorem is also true with the $\widetilde{\mathcal{H}}$,$\widetilde{\mathcal{P}}$.
\end{remark}

\section{Continuity of the $\widetilde{\mathcal{H}}_{(C,D)}$ in $D$}

Now we will transform the basis to show its continuity in $D$ at $0$, that means when $D\longrightarrow 0$ the basis tends to be the one for $D=0$

\begin{enumerate}
\item $C\neq 0$\\

\begin{tabular}{l | l}
\begin{minipage}{6cm}
$D >0$\\
Let's define :\\

$\begin{array}{l}
\widetilde{R}_1  :=-(\dfrac{\widetilde{P}_2-\widetilde{P}_1}{\varepsilon}+\dfrac{2}{\varepsilon^2}\widetilde{P}_3),\\
\widetilde{R}_2  :=\dfrac{1}{2}(\widetilde{P}_1+\widetilde{P}_2- \dfrac{1}{2}\theta^2 \widetilde{P}_4),\\
\widetilde{R}_3  :=\widetilde{P}_3,\\
\widetilde{R}_4  :=\widetilde{P}_4.\\
\end{array}$\\

Using a Taylor expansion of order $2$ in $\varepsilon$ near to $0$ above 0, we find :\\

$\begin{array}{l}
\widetilde{R}_1\underset{\varepsilon \rightarrow 0}{\rightarrow}\widetilde{M}_1,\\
\widetilde{R}_2\underset{\varepsilon \rightarrow 0}{\rightarrow}\widetilde{M}_2,\\
\widetilde{R}_3 =\widetilde{M}_3,\\
\widetilde{R}_4 =\widetilde{M}_4.
\end{array}$

\end{minipage}
&
\begin{minipage}{6cm}
 $D <0$\\
Let's define :\\

$\begin{array}{l}
\widetilde{V}_1  :=2(\dfrac{\widetilde{P}_2}{\varepsilon}+\dfrac{1}{\varepsilon^2}\widetilde{P}_3),\\
\widetilde{V}_2  :=\widetilde{P}_1-\dfrac{\theta^2}{4} \widetilde{P}_4,\\
\widetilde{V}_3  :=\widetilde{P}_3,\\
\widetilde{V}_4  :=\widetilde{P}_4.\\
\end{array}$\\

Using a Taylor expansion of order $2$ in $\varepsilon$ near to $0$ above zero, we find :\\

$\begin{array}{l}
\widetilde{V}_1\underset{\varepsilon \rightarrow 0}{\rightarrow}\widetilde{M}_1,\\
\widetilde{V}_2\underset{\varepsilon \rightarrow 0}{\rightarrow}\widetilde{M}_2,\\
\widetilde{V}_3 =\widetilde{M}_3,\\
\widetilde{V}_4 =\widetilde{M}_4.
\end{array}$
\end{minipage}
\end{tabular}

\item $C= 0$\\
We use again the same changes for the 4 first vectors in the 2 cases.\\

\begin{tabular}{l|l}
\begin{minipage}{6cm}
$D >0$ \\
Let's define :\\

$\begin{array}{l}
\widetilde{R}_5  :=\widetilde{P}_5-\widetilde{P}_6,\\
\widetilde{R}_6  :=\widetilde{P}_5.\\
\end{array}$\\

Using a Taylor expansion of order $2$ in $\varepsilon$ near to $0$ above zero, we find :  \\

$\begin{array}{l}
\widetilde{R}_5 \underset{\varepsilon \rightarrow 0}{\rightarrow}\widetilde{M}_5,\\
\widetilde{R}_6 \underset{\varepsilon \rightarrow 0}{\rightarrow}\widetilde{M}_6.
\end{array}$
\end{minipage}
 &
\begin{minipage}{6cm}
$D<0$\\
Let's define :\\

$\begin{array}{l}
\widetilde{V}_5  :={2{\varepsilon}\widetilde{P}_6},\\
\widetilde{V}_6  :=\widetilde{P}_5.\\
\end{array}$\\

Using a Taylor expansion of order $2$ in $\varepsilon$ near to $0$ above zero, we find :\\

$\begin{array}{l}
\widetilde{V}_5 \underset{\varepsilon \rightarrow 0}{\rightarrow}\widetilde{M}_5,\\
\widetilde{V}_6 \underset{\varepsilon \rightarrow 0}{\rightarrow}\widetilde{M}_6.\\
\end{array}$
\end{minipage} 
\end{tabular}
\end{enumerate}

\section{Computation of the Lie brackets of $\widetilde{\mathcal{H}}_{(C,D)}$}

\begin{enumerate}

\item $D>0$.
\begin{equation}
\begin{split}
[\widetilde{R}_1,\widetilde{R}_2]&=-\frac{1}{\varepsilon^2}(e^{\varepsilon t}+e^{-\varepsilon t}-2)\frac{\partial}{\partial t}+\frac{1}{2\varepsilon}q (e^{-\varepsilon t}-e^{\varepsilon t})\frac{\partial}{\partial q}\frac{q^2}{4\gamma}(e^{\varepsilon t}+e^{-\varepsilon t})-\frac{1}{4\varepsilon}(e^{-\varepsilon t}-e^{\varepsilon t})=\widetilde{R}_1\\
[\widetilde{R}_1,\widetilde{R}_3]&=\frac{e^{-\varepsilon t}-e^{\varepsilon t}}{\varepsilon}\frac{\partial}{\partial t}-\frac{e^{\varepsilon t}+e^{-\varepsilon t}}{2}q\frac{\partial}{\partial q}+\frac{q^2}{4\gamma}\varepsilon (e^{\varepsilon t}-e^{-\varepsilon t})-\frac{1}{4}(e^{\varepsilon t}+e^{-\varepsilon t})=2\widetilde{R}_2+\frac{\gamma}{2}\widetilde{R}_4\\
[\widetilde{R}_1,\widetilde{R}_4]&=0\\
[\widetilde{R}_1,\widetilde{R}_5]&=0\\
[\widetilde{R}_1,\widetilde{R}_6]&=\frac{1}{\varepsilon}(e^{-\frac{\varepsilon}{2} t}-e^{\frac{\varepsilon}{2} t })\frac{\partial}{\partial q}+\frac{q}{2\gamma}(e^{-\frac{\varepsilon}{2} t}+e^{\frac{\varepsilon}{2} t})=\widetilde{R}_5\\
[\widetilde{R}_2,\widetilde{R}_3]&=-\frac{1}{2}(e^{\varepsilon t}+e^{-\varepsilon t})\frac{\partial}{\partial t}-\frac{1}{4}q \varepsilon (e^{\varepsilon t}-e^{-\varepsilon t})\frac{\partial}{\partial q}+\frac{q^2}{8\gamma}\varepsilon^2(e^{\varepsilon t}+e^{-\varepsilon t})-\frac{1}{8}\varepsilon(e^{\varepsilon t}-e^{-\varepsilon t})=\frac{1}{2}\varepsilon^2\widetilde{R}_1+\widetilde{R}_3 \\
[\widetilde{R}_2,\widetilde{R}_4]&=0\\
[\widetilde{R}_2,\widetilde{R}_5]&=\frac{1}{2\varepsilon}(e^{\frac{\varepsilon}{2} t}-e^{-\frac{\varepsilon}{2} t})\frac{\partial}{\partial q}-\frac{1}{4\gamma}(e^{\frac{\varepsilon}{2} t}+e^{-\frac{\varepsilon}{2} t})=-\frac{1}{2}\widetilde{R}_5\\
[\widetilde{R}_2,\widetilde{R}_6]&=-\frac{1}{2}e^{\frac{\varepsilon}{2} t}\frac{\partial}{\partial q}+\frac{1}{2\gamma}q\frac{\varepsilon}{2} e^{\frac{\varepsilon}{2} t}=\frac{\varepsilon}{2} \widetilde{R}_5+\frac{1}{2}\widetilde{R}_6\\
[\widetilde{R}_3,\widetilde{R}_4]&=0\\
[\widetilde{R}_3,\widetilde{R}_5]&=(e^{\frac{\varepsilon}{2} t}+e^{-\frac{\varepsilon}{2} t})\frac{\partial}{\partial q}-\frac{1}{2\gamma}q\frac{\varepsilon}{2}(e^{\frac{\varepsilon}{2} t}-e^{-\frac{\varepsilon}{2} t})=-\frac{\varepsilon}{2} \widetilde{R}_5-\widetilde{R}_6\\
[\widetilde{R}_3,\widetilde{R}_6]&=-\frac{\varepsilon}{2} e^{-\frac{\varepsilon}{2} t}\frac{\partial}{\partial q}-\frac{1}{\gamma}q\frac{\varepsilon}{2}^2 e^{-\frac{\varepsilon}{2} t}=\frac{\varepsilon}{2} \widetilde{R}_6\\
[\widetilde{R}_4,\widetilde{R}_5]&=0\\
[\widetilde{R}_4,\widetilde{R}_6]&=0\\
[\widetilde{R}_5,\widetilde{R}_6]&=\frac{1}{\gamma}=-\widetilde{R}_4
\end{split}
\end{equation}

\item $D=0$.
\begin{equation}
\begin{split}
[\widetilde{M}_1,\widetilde{M}_2]&=  \widetilde{M}_1\\
[\widetilde{M}_1,\widetilde{M}_3]&=2\widetilde{M}_2+\frac{\gamma}{2}\widetilde{M}_4\\
[\widetilde{M}_1,\widetilde{M}_4]&=0\\
[\widetilde{M}_1,\widetilde{M}_5]&=0\\
[\widetilde{M}_1,\widetilde{M}_6]&=-t\frac{\partial}{\partial q}+\frac{1}{\gamma}q=\widetilde{M}_5\\
[\widetilde{M}_2,\widetilde{M}_3]&=\widetilde{M}_3\\
[\widetilde{M}_2,\widetilde{M}_4]&=0\\
[\widetilde{M}_2,\widetilde{M}_5]&=\frac{1}{2}t\frac{\partial}{\partial q}-\frac{1}{2\gamma}q=-\frac{1}{2}\widetilde{M}_5\\
[\widetilde{M}_2,\widetilde{M}_6]&=-\frac{1}{2}\frac{\partial}{\partial q}=\frac{1}{2}\widetilde{M}_6\\
[\widetilde{M}_3,\widetilde{M}_4]&=0\\
[\widetilde{M}_3,\widetilde{M}_5]&=\frac{\partial}{\partial q}=-\widetilde{M}_6\\
[\widetilde{M}_3,\widetilde{M}_6]&=0\\
[\widetilde{M}_4,\widetilde{M}_5]&=0\\
[\widetilde{M}_4,\widetilde{M}_6]&=0\\
[\widetilde{M}_5,\widetilde{M}_6]&=\frac{1}{\gamma}=-\widetilde{M}_4\\
\end{split}
\end{equation}
\begin{remark}
Now we compute the Lie brackets of the $\widetilde{N_*}$ from Eq(\ref{ntilde})
\begin{equation}
\begin{split}
[\widetilde{N}_1,\widetilde{N}_2]&=[\widetilde{M}_3,\widetilde{M}_6]=0\\
[\widetilde{N}_1,\widetilde{N}_3]&=[\widetilde{M}_3,-\gamma\widetilde{M}_4]=0\\
[\widetilde{N}_1,\widetilde{N}_4]&=[\widetilde{M}_3,2\widetilde{M}_2]=-2\widetilde{M}_3=-2\widetilde{N}_1\\
[\widetilde{N}_1,\widetilde{N}_5]&=[\widetilde{M}_3,-2\widetilde{M}_5]=2\widetilde{M}_6=2\widetilde{N}_2\\
[\widetilde{N}_1,\widetilde{N}_6] &=[\widetilde{M}_3,2\widetilde{M}_1]=-4\widetilde{M}_2-\gamma\widetilde{M}_4= -2\widetilde{N}_4+\widetilde{N}_3\\
[\widetilde{N}_2,\widetilde{N}_3]&=[\widetilde{M}_6,-\gamma\widetilde{M}_4]=0\\
[\widetilde{N}_2,\widetilde{N}_4]&=[\widetilde{M}_6,2\widetilde{M}_2]=-\widetilde{M}_6=-\widetilde{N}_2\\
[\widetilde{N}_2,\widetilde{N}_5]&=[\widetilde{M}_6,-2\widetilde{M}_5]=-2\widetilde{M}_4=\frac{2}{\gamma}\widetilde{N}_3\\
[\widetilde{N}_2,\widetilde{N}_6]&=[\widetilde{M}_6,2\widetilde{M}_1]=-2\widetilde{M}_5= \widetilde{N}_5\\
[\widetilde{N}_3,\widetilde{N}_4]&=[-\gamma \widetilde{M}_4,2\widetilde{M}_2]=0\\
[\widetilde{N}_3,\widetilde{N}_5]&=[-\gamma \widetilde{M}_4,-2\widetilde{M}_5]=0\\
[\widetilde{N}_3,\widetilde{N}_6]&=[-\gamma \widetilde{M}_4,2\widetilde{M}_1]=0\\
[\widetilde{N}_4,\widetilde{N}_5]&=[2\widetilde{M}_2,-2\widetilde{M}_5]=2\widetilde{M}_5= -\widetilde{N}_5\\
[\widetilde{N}_4,\widetilde{N}_6]&=[2\widetilde{M}_2,2\widetilde{M}_1]=-4\widetilde{M}_1= -2\widetilde{N}_6\\
[\widetilde{N}_5,\widetilde{N}_6]&=[-2\widetilde{M}_5,2\widetilde{M}_1]=0\\
\end{split}
\end{equation} 
That is consistent with the results of \cite{LZ2} ,\S 3, and the fact that $N\mapsto -\widetilde{N}$ is an algebra morphism (\textit{Lemma \ref{HJB.defNtilde}}).
\end{remark}
\item $D<0$.

\begin{equation}
\begin{split}
[\widetilde{V}_1,\widetilde{V}_2]&=(-\frac{2}{\varepsilon^2}+\frac{2}{\varepsilon\gamma}\cos(\varepsilon t))\frac{\partial}{\partial t} -\frac{q}{\varepsilon}\sin(\varepsilon t)\frac{\partial}{\partial q}+\frac{ \cos(\varepsilon t) q^2}{2\gamma}-\frac{\sin(\varepsilon t) }{2\varepsilon}=\widetilde{V}_1\\
[\widetilde{V}_1,\widetilde{V}_3]&=-\frac{2}{\varepsilon}\sin(\varepsilon t) \frac{\partial}{\partial t}-q \cos(\varepsilon t) \frac{\partial}{\partial q}-\frac{1}{2\varepsilon \gamma} q^2\sin(\varepsilon t)-\frac{1}{2}\cos(\varepsilon t)=2\widetilde{V}_2+\frac{\gamma}{2} \widetilde{V}_4\\
[\widetilde{V}_1,\widetilde{V}_4]&=0\\
[\widetilde{V}_1,\widetilde{V}_5]&=0\\
[\widetilde{V}_1,\widetilde{V}_6]&=-\frac{2}{\varepsilon}\sin(\frac{\varepsilon}{2} t)\frac{\partial}{\partial q}+\frac{q}{\gamma}\cos(\frac{\varepsilon}{2} t)=\widetilde{V}_5\\
[\widetilde{V}_2,\widetilde{V}_3]&=-\cos(\varepsilon t) \frac{\partial}{\partial t}+q\frac{1}{2}\varepsilon \sin(\varepsilon t)\frac{\partial}{\partial q}-\frac{1}{4 \gamma}q^2\varepsilon^2 \cos(\varepsilon t)+\frac{1}{4}\varepsilon \sin(\varepsilon t)=-\frac{\varepsilon^2}{2}\widetilde{V}_1+\widetilde{V}_3\\
[\widetilde{V}_2,\widetilde{V}_4]&=0\\
[\widetilde{V}_2,\widetilde{V}_5]&=\frac{1}{\varepsilon}\sin(\frac{\varepsilon}{2} t) \frac{\partial}{\partial q}-\frac{1}{2\gamma}q\cos(\frac{\varepsilon}{2} t)=-\frac{1}{2}\widetilde{V}_5\\
[\widetilde{V}_2,\widetilde{V}_6]&=-\frac{1}{2}\cos(\frac{\varepsilon}{2} t) \frac{\partial}{\partial q}-\frac{\varepsilon}{4\gamma}q\sin(\frac{\varepsilon}{2} t)=\frac{1}{2}\widetilde{V}_6\\
[\widetilde{V}_3,\widetilde{V}_4]&=0\\
[\widetilde{V}_3,\widetilde{V}_5]&=\cos(\frac{\varepsilon}{2} t) \frac{\partial}{\partial q}+\frac{\varepsilon}{2\gamma}q\sin(\frac{\varepsilon}{2} t)=-\widetilde{V}_6\\
[\widetilde{V}_3,\widetilde{V}_6]&=-\frac{\varepsilon}{2} \sin(\frac{\varepsilon}{2} t)\frac{\partial}{\partial q}+\frac{1}{\gamma}q\frac{\varepsilon}{2} \cos(\frac{\varepsilon}{2} t)=\frac{\varepsilon}{2}^2 \widetilde{V}_5\\
[\widetilde{V}_4,\widetilde{V}_5]&=0\\
[\widetilde{V}_4,\widetilde{V}_6]&=0\\
[\widetilde{V}_5,\widetilde{V}_6]&=\frac{1}{\gamma}=-\widetilde{V}_4\\
\end{split}
\end{equation}

\end{enumerate}

\section[Isomorphisms and Structure]{Isomorphism of $\widetilde{\mathcal{H}}_{(C,D)}$ and $\widetilde{\mathcal{H}}_{(C,0)}$ with $C \neq 0$ and of $\widetilde{\mathcal{H}}_{(0,D)}$ and $\widetilde{\mathcal{H}}_{(0,0)}$. Structure of the algebra. }

\subsection{ $D>0$}

The two algebras are isomorphic. To prove it, we use a different basis of $\widetilde{\mathcal{H}}_{(C,D)}$ :
\begin{equation}
\begin{split}
S_1&:=\widetilde{R}_1\\
S_2&:=-\frac{\varepsilon}{2} \widetilde{R}_1+\widetilde{R}_2\\
S_3&:=\varepsilon^2 \widetilde{R}_1 -\varepsilon \widetilde{R}_2+\widetilde{R}_3 - \frac{\varepsilon}{2} \widetilde{R}_4\\
S_4&:=\widetilde{R}_4\\
S_5&:=\widetilde{R}_5\\
S_6&:=\widetilde{R}_6\\
\end{split}
\end{equation}
If $C\neq 0$, we define $\widetilde{S}_5$ and $\widetilde{S}_6$ as equal to zero.

Then we define the isomorphism :

\begin{equation}
\begin{split}
\phi : \widetilde{\mathcal{H}}_{(C,D)} &\longrightarrow \widetilde{\mathcal{H}}_{(C,0)}\\
 S_i &\longmapsto M_i
\end{split}
\end{equation}

We note that the 4 first vectors are independent from the last two, so that the two cases $C=0$ and $C\neq 0$ are here.

\subsection{ $D<0$}

Let's define a new basis :
\begin{equation}
\begin{split}
S_1&:=\widetilde{V}_1\\
S_2&:=\widetilde{V}_2\\
S_3&:=-\frac{\varepsilon}{2}^2 \widetilde{V}_1+\widetilde{V}_3\\
S_4&:=\widetilde{V}_4\\
S_5&:=\widetilde{V}_5\\
S_6&:=\widetilde{V}_6.\\
\end{split}
\end{equation}
if $C\neq 0$, we define $\widetilde{S}_5$ and $\widetilde{S}_6$ as equal to zero.

Then we define the isomorphism :

\begin{equation}
\begin{split}
\phi : \widetilde{\mathcal{H}}_{(C,D)} &\longrightarrow \widetilde{\mathcal{H}}_{(C,0)}\\
 V_i &\longmapsto M_i.
\end{split}
\end{equation}

Then again, we note that the 4 first vectors are independent from the last two, so the two cases $C=0$ and $C\neq 0$ are here.

\subsection{Structure}

\begin{theorem}
\begin{enumerate}
\item  $\widetilde{H}_{(0,D)}$ is isomorphic to the semidirect product of $sl_2(\mathbb{R})$ and a Heisenberg algebra $H_3$ of dimension $3$.
\item $\widetilde{H}_{(C,D)}$, $C\neq 0$, is isomorphic to the direct product of the center $Z(H_3)\simeq \mathbb{R}$ and $sl_2(\mathbb{R})$; in fact the only isovector remaining from the Heisenberg algebra when $C \neq 0$ generates the center of
 $\widetilde{H}_{(0,D)}$.
\end{enumerate}
\end{theorem}

\begin{proof}
Using the isomorphisms defined above, it is sufficient to show the theorem for the basis $\widetilde{M}_*$.\\
Setting :
\begin{align*}
e&:=-\dfrac{1}{2}\widetilde{M}_3\\
f&:=2\widetilde{M}_1\\
h&:=2\widetilde{M}_2+\dfrac{\gamma}{2}\widetilde{M}_4,
\end{align*}
we have the usual $sl_2(\mathbb{R})$ basis which acts on $\langle \widetilde{M}_4, \widetilde{M}_5, \widetilde{M}_6 \rangle \simeq H_3$, with $H_3$ the Heisenberg algebra of dimension $3$, and $Z(H_3)\simeq \langle \widetilde{M}_4 \rangle$.

\end{proof}

\begin{remark} This precises the identification made in \S 3 in \cite{LZ2}.
\end{remark}

\section{Finding the mapping from the isovectors}

The aim is now to find the mapping associated with the isovectors found in the case $D=0$ and $C=0$, to compute explicitly $e^{\mu \widetilde{M}_*}$ This associates to a solution $\eta$ of the Eq(\ref{eqdep}) another solution : 
\begin{equation}
\frac{d}{d\mu} (e^{\mu \widetilde{M}_*}\eta(t,q))_{\mid \mu=0}=\widetilde{M}_*\eta(t,q)
\end{equation}
and
\begin{equation}
e^{\mu' \widetilde{M}_*}(e^{\mu \widetilde{M}_*}\eta(t,q))=e^{(\mu+\mu') \widetilde{M}_*}\eta(t,q);
\end{equation}
these properties characterize $e^{\mu \widetilde{M}_*}$.\\
One finds

\begin{equation}
\begin{split}
\widetilde{M}_1=-tq \frac{\partial}{\partial q}-t^2\frac{\partial}{\partial t}-\frac{\gamma t-q^2}{2\gamma} \cdot \qquad \qquad &e^{\mu \widetilde{M}_1}\eta(t,q)=\frac{1}{\sqrt{1+\mu t}}e^{\frac{\mu q^2}{2\gamma (1+\mu t) }}\eta(\frac{t}{1+\mu t}, \frac{q}{1+\mu t})\\
\widetilde{M}_2=-\frac{q}{2}\frac{\partial}{\partial q}-t \frac{\partial}{\partial t}\qquad \qquad &e^{\mu \widetilde{M}_2}\eta(t,q)=\eta(e^{-\mu}t,e^{-\frac{\mu}{2}}q) \\
\widetilde{M}_3=-\frac{\partial}{\partial t} \qquad \qquad &e^{\mu \widetilde{M}_3}\eta(t,q)=\eta(t-\mu,q)\\
\widetilde{M}_4=-\frac{1}{\gamma}\cdot \qquad \qquad &e^{\mu \widetilde{M}_4}\eta(t,q)=e^{-\frac{\mu}{\gamma}}\eta(t,q) \\
\widetilde{M}_5=-t\frac{\partial}{\partial q}+\frac{1}{\gamma}q \cdot \qquad \qquad &e^{\mu \widetilde{M}_5}\eta(t,q)=e^{\frac{\mu}{\gamma}q-\frac{\mu^2 t}{2\gamma}}\eta(t,q-\mu t)\\
\widetilde{M}_6=-\frac{\partial}{\partial q} \qquad \qquad &e^{\mu \widetilde{M}_6}\eta(t,q)=\eta(t,q-\mu).
\end{split}
\end{equation}
The computations of $e^{\mu \widetilde{M}_1}$ and $e^{\mu \widetilde{M}_2}$ had already been made in \cite{LZ2} \S 7.
This list is given only for completeness. It can be found, for instance in \cite{OPJ}.

\section{Computation of $\Omega_\eta$, on the isovectors basis}
Using the operator from \textit{Corollary \ref{defOmega}},
\begin{equation}
\begin{split}
\Omega_\eta :\mathcal{H}^2_{(C,D)} &\rightarrow \bigwedge T^*(\mathbb{R}^2)\\
(\widetilde{N},\widetilde{N}')&\mapsto \frac{\gamma}{\eta}([\widetilde{N},\widetilde{N}']^t\frac{\partial}{\partial t}+[\widetilde{N},\widetilde{N}']^q\frac{\partial}{\partial q}+\frac{1}{\gamma}[\widetilde{N},\widetilde{N}']^S)\eta
\end{split}
\end{equation}
and compute its value on the different vectors of the basis given previously.\\
Let's define 
\begin{equation}
\begin{split}
\widetilde{B}&=-\frac{\gamma}{\eta}\frac{\partial \eta}{\partial q}\\
\widetilde{E}&=-\frac{\gamma}{\eta}\frac{\partial \eta}{\partial t}
\end{split}
\end{equation}

\begin{enumerate}
\item $ D>0$.

\begin{equation}
\begin{split}
\Omega_\eta(\widetilde{R}_1,\widetilde{R}_2)&=\frac{\gamma}{\varepsilon^2}(e^{\varepsilon t}+e^{-\varepsilon t}-2)\widetilde{E}-\frac{\gamma}{ 2\varepsilon}q (e^{-\varepsilon t}-e^{\varepsilon t})\widetilde{B}\\
&\qquad+\frac{q^2}{4}(e^{\varepsilon t}+e^{-\varepsilon t})-\frac{1}{4\varepsilon}(e^{-\varepsilon t}-e^{\varepsilon t})\\
\Omega_\eta(\widetilde{R}_1,\widetilde{R}_3)&=-\gamma\frac{e^{-\varepsilon t}-e^{\varepsilon t}}{\varepsilon}\widetilde{E}+\gamma \frac{e^{\varepsilon t}+e^{-\varepsilon t}}{\eta}q\widetilde{B}+\frac{q^2}{4}\varepsilon (e^{\varepsilon t}-e^{-\varepsilon t})-\frac{\gamma}{4}(e^{\varepsilon t}+e^{-\varepsilon t})\\
\Omega_\eta(\widetilde{R}_1,\widetilde{R}_4)&=0\\
\Omega_\eta(\widetilde{R}_1,\widetilde{R}_5)&=0\\
\Omega_\eta(\widetilde{R}_1,\widetilde{R}_6)&=-\frac{\gamma}{\varepsilon}(e^{-\frac{\varepsilon}{2} t}-e^{\frac{\varepsilon}{2} t })\widetilde{B}+\frac{q}{2}(e^{-\frac{\varepsilon}{2} t}+e^{\frac{\varepsilon}{2} t})\\
\Omega_\eta(\widetilde{R}_2,\widetilde{R}_3)&=\frac{\gamma}{2}(e^{\varepsilon t}+e^{-\varepsilon t})\widetilde{E}+\frac{\gamma}{4}q \varepsilon (e^{\varepsilon t}-e^{-\varepsilon t})\widetilde{B}+\frac{q^2}{8}\varepsilon^2(e^{\varepsilon t}+e^{-\varepsilon t})-\frac{\gamma}{8}\varepsilon(e^{\varepsilon t}-e^{-\varepsilon t})\\
\Omega_\eta(\widetilde{R}_2,\widetilde{R}_4)&=0\\
\Omega_\eta(\widetilde{R}_2,\widetilde{R}_5)&=-\frac{\gamma}{2\varepsilon}(e^{\frac{\varepsilon}{2} t}-e^{-\frac{\varepsilon}{2} t})\widetilde{B}-\frac{1}{4}(e^{\frac{\varepsilon}{2} t}+e^{-\frac{\varepsilon}{2} t})\\
\Omega_\eta(\widetilde{R}_2,\widetilde{R}_6)&=\frac{\gamma}{2}e^{\frac{\varepsilon}{2} t}\widetilde{B}+\frac{1}{2}q\frac{\varepsilon}{2} e^{\frac{\varepsilon}{2} t}\\
\Omega_\eta(\widetilde{R}_3,\widetilde{R}_4)&=0\\
\Omega_\eta(\widetilde{R}_3,\widetilde{R}_5)&=-\gamma(e^{\frac{\varepsilon}{2} t}+e^{-\frac{\varepsilon}{2} t})\widetilde{B}-\frac{1}{2}q\frac{\varepsilon}{2}(e^{\frac{\varepsilon}{2} t}-e^{-\frac{\varepsilon}{2} t})\\
\Omega_\eta(\widetilde{R}_3,\widetilde{R}_6)&=\gamma\frac{\varepsilon}{2} e^{-\frac{\varepsilon}{2} t}\widetilde{B}-q\frac{\varepsilon}{2}^2 e^{-\frac{\varepsilon}{2} t}\\
\Omega_\eta(\widetilde{R}_4,\widetilde{R}_5)&=0\\
\Omega_\eta(\widetilde{R}_4,\widetilde{R}_6)&=0\\
\Omega_\eta(\widetilde{R}_5,\widetilde{R}_6)&=1\\
\end{split}
\end{equation}

\item $D=0$.
\begin{equation}
\begin{split}
\Omega_\eta(\widetilde{M}_1,\widetilde{M}_2)&=\gamma t^2\widetilde{E}+\gamma\widetilde{B}-\frac{t\gamma-q^2}{2}\\
\Omega_\eta(\widetilde{M}_1,\widetilde{M}_3)&=\gamma q \widetilde{B}+\gamma\eta t\widetilde{E}-\frac{\gamma}{2}\\
\Omega_\eta(\widetilde{M}_1,\widetilde{M}_4)&=0\\
\Omega_\eta(\widetilde{M}_1,\widetilde{M}_5)&=0\\
\Omega_\eta(\widetilde{M}_1,\widetilde{M}_6)&=\gamma t\widetilde{B}+q\\
\Omega_\eta(\widetilde{M}_2,\widetilde{M}_3)&=\gamma\widetilde{E}\\
\Omega_\eta(\widetilde{M}_2,\widetilde{M}_4)&=0\\
\Omega_\eta(\widetilde{M}_2,\widetilde{M}_5)&=-\frac{\gamma}{2}t\widetilde{B}-\frac{1}{2}q\\
\Omega_\eta(\widetilde{M}_2,\widetilde{M}_6)&=\frac{\gamma}{2}\widetilde{B}\\
\Omega_\eta(\widetilde{M}_3,\widetilde{M}_4)&=0\\
\Omega_\eta(\widetilde{M}_3,\widetilde{M}_5)&=-\gamma \widetilde{B}\\
\Omega_\eta(\widetilde{M}_3,\widetilde{M}_6)&=0\\
\Omega_\eta(\widetilde{M}_4,\widetilde{M}_5)&=0\\
\Omega_\eta(\widetilde{M}_4,\widetilde{M}_6)&=0\\
\Omega_\eta(\widetilde{M}_5,\widetilde{M}_6)&=1\\
\end{split}
\end{equation}

\item $D<0$.
\begin{equation}
\begin{split}
\Omega_\eta(\widetilde{V}_1,\widetilde{V}_2)&=(\frac{2\gamma}{\varepsilon^2}-\frac{2\gamma}{\varepsilon^2}\cos(\varepsilon t))\widetilde{E} +\frac{q\gamma}{\varepsilon}\sin(\varepsilon t)\widetilde{B}+\frac{ \cos(\varepsilon t) q^2}{2}-\frac{\gamma\sin(\varepsilon t) }{2\varepsilon}\\
\Omega_\eta(\widetilde{V}_1,\widetilde{V}_3)&=\frac{2\gamma}{\varepsilon }\sin(\varepsilon t) \widetilde{E}+\gamma q \cos(\varepsilon t) \widetilde{B}-\frac{1}{2\varepsilon} q^2\sin(\varepsilon t)-\frac{\gamma}{2}\cos(\varepsilon t)\\
\Omega_\eta(\widetilde{V}_1,\widetilde{V}_4)&=0\\
\Omega_\eta(\widetilde{V}_1,\widetilde{V}_5)&=0\\
\Omega_\eta(\widetilde{V}_1,\widetilde{V}_6)&=\frac{2 \gamma}{\varepsilon}\sin(\frac{\varepsilon}{2} t)\widetilde{B}+q\cos(\frac{\varepsilon}{2} t)\\
\Omega_\eta(\widetilde{V}_2,\widetilde{V}_3)&=\gamma\cos(\varepsilon t) \widetilde{E}-q\frac{\gamma}{2}\varepsilon \sin(\varepsilon t)\widetilde{B}-\frac{1}{4}q^2\varepsilon^2 \cos(\varepsilon t)+\frac{1}{4}\varepsilon \sin(\varepsilon t)\\
\Omega_\eta(\widetilde{V}_2,\widetilde{V}_4)&=0\\
\Omega_\eta(\widetilde{V}_2,\widetilde{V}_5)&=-\frac{\gamma}{\varepsilon}\sin(\frac{\varepsilon}{2} t) \widetilde{B}-\frac{1}{2}q\cos(\frac{\varepsilon}{2} t)\\
\Omega_\eta(\widetilde{V}_2,\widetilde{V}_6)&=\frac{\gamma}{2}\cos(\frac{\varepsilon}{2} t) \widetilde{B}-\frac{\varepsilon}{4}q\sin(\frac{\varepsilon}{2} t)\\
\Omega_\eta(\widetilde{V}_3,\widetilde{V}_4)&=0\\
\Omega_\eta(\widetilde{V}_3,\widetilde{V}_5)&=-\gamma\cos(\frac{\varepsilon}{2} t) \widetilde{B}+\frac{\varepsilon}{2} q\sin(\frac{\varepsilon}{2} t)\\
\Omega_\eta(\widetilde{V}_3,\widetilde{V}_6)&=\gamma\frac{\varepsilon}{2} \sin(\frac{\varepsilon}{2} t)\widetilde{B}+q\frac{\varepsilon}{2} \cos(\frac{\varepsilon}{2} t)\\
\Omega_\eta(\widetilde{V}_4,\widetilde{V}_5)&=0\\
\Omega_\eta(\widetilde{V}_4,\widetilde{V}_6)&=0\\
\Omega_\eta(\widetilde{V}_5,\widetilde{V}_6)&=1\\
\end{split}
\end{equation}
\end{enumerate}

Using the fact that $\Omega_\eta(N,N')(t,z(t))$ is a martingale for the filtration of the brownian motion $w$ (see \S 4), we find new martingales.

\section{Explicit determination of the $\widetilde{\mathcal{J}}_{(C,D)}$ and $\widetilde{\mathcal{K}}_{(C,D)}$ }

From \textit{Lemmas \ref{Jv-Kv}} and \textit{\ref{HJB.lemmeNtilde}} we deduce that $\widetilde{\mathcal{K}}_V=\{ \widetilde{N} \mid N \in \mathcal{K}_V \}$ and $\widetilde{\mathcal{J}}_V=\{\widetilde{N} \mid N \in\mathcal{H}_V \}$ are subalgebras of $\widetilde{\mathcal{H}}_V=\{\widetilde{N} \mid N\in \mathcal{H}_V \}$.
Let's find a basis for $\widetilde{\mathcal{K}}_V$ and  $\widetilde{\mathcal{J}}_V$ in the case of this potential. Setting $\widetilde{\mathcal{K}}_{(C,D)}=\widetilde{\mathcal{K}}_V$ and $\widetilde{\mathcal{J}}_{(C,D)}=\widetilde{\mathcal{J}}_{V}$\\

\begin{tabular}{|c |c |c |}
\hline
 			& $C\neq 0$ 	& $C=0$ \\
 \hline
 $D<0$ or $D>0 $	 	& 
 \begin{minipage}{4cm}
 \bigskip
 	$\widetilde{\mathcal{J}}_{(C,D)}=\{\widetilde{P}_3,\widetilde{P}_4\}$\\
 	$\widetilde{\mathcal{K}}_{(C,D)}=\{\widetilde{P}_1,\widetilde{P}_2,\widetilde{P}_3,\widetilde{P}_4\}$\\
 	$\widetilde{\mathcal{J}}_{(C,D)} \cap \widetilde{\mathcal{K}}_{(C,D)} =\{\widetilde{P}_3,\widetilde{P}_4\}$
 \bigskip
 \end{minipage}
 &
 \begin{minipage}{4cm}
 \bigskip
 	$\widetilde{\mathcal{J}}_{(C,D)}=\{\widetilde{P}_3,\widetilde{P}_4\}$\\
 	$\widetilde{\mathcal{K}}_{(C,D)}=\{\widetilde{P}_1,\widetilde{P}_2,\widetilde{P}_3,\widetilde{P}_4\}$\\
 	$\widetilde{\mathcal{J}}_{(C,D)} \cap \widetilde{\mathcal{K}}_{(C,D)} =\{\widetilde{P}_3,\widetilde{P}_4\}$
 \bigskip
 \end{minipage}
 \\
 \hline

 $D=0$		& 
  \begin{minipage}{5cm}
  \bigskip
 	$\widetilde{\mathcal{J}}_{(C,D)}=\{\widetilde{M}_2,\widetilde{M}_3,\widetilde{M}_4\}$\\
 	$\widetilde{\mathcal{K}}_{(C,D)}=\{\widetilde{M}_1,\widetilde{M}_2,\widetilde{M}_3,\widetilde{M}_4\}$\\
 	$\widetilde{\mathcal{J}}_{(C,D)} \cap \widetilde{\mathcal{K}}_{(C,D)} =\{\widetilde{M}_2,\widetilde{M}_3,\widetilde{M}_4\}$
 \bigskip
 \end{minipage}
 &
  \begin{minipage}{5cm}
  \bigskip
 	$\widetilde{\mathcal{J}}_{(C,D)}=\{\widetilde{M}_2,\widetilde{M}_3,\widetilde{M}_4, \widetilde{M}_6\}$\\
 	$\widetilde{\mathcal{K}}_{(C,D)}=\{\widetilde{M}_1,\widetilde{M}_2,\widetilde{M}_3,\widetilde{M}_4\}$\\
 	$\widetilde{\mathcal{J}}_{(C,D)} \cap \widetilde{\mathcal{K}}_{(C,D)} =\{\widetilde{M}_2,\widetilde{M}_3,\widetilde{M}_4\}$
 \bigskip
 \end{minipage}\\
 \hline
\end{tabular}
\\

\begin{remark} We shall notice that in the case $C=0$, $\widetilde{\mathcal{K}}_{C,D}=\widetilde{\mathcal{P}}_D$ ; this is a direct consequence of Lemma 3.8 of \cite{theseHouda}.\\
We see that their intersection does not depend on $C$ but only on $D $.
\end{remark}


\section{Isovectors and affine models}
We now come back to the situation described in \S \ref{aff1} \begin{proposition} The isovector algebra $\mathcal H_{V}$ associated with $V$ has dimension $6$ if and only if $\widetilde{\phi}\in\{\frac{\alpha}{4},\frac{3\alpha}{4}\}$, i.e. $\delta\in \{1,3\}$ ; in the opposite case,
it has dimension $4$.
\end{proposition}
\begin{proof}
It is enough to apply the last result from \S \ref{Isovectors}, observing that the condition $C=0$ is equivalent to $\widetilde{\phi}\in \{\frac{\alpha}{4},\frac{3\alpha}{4}\}$.
\end{proof}

In the context of H\'enon's already mentioned PhD thesis (\cite{Henon}, p.55)
we have \newline $\phi=\kappa a$, $\lambda=\kappa$, $\alpha=\sigma^{2}$ et $\beta=0$,
whence $\widetilde{\phi}=\kappa a$ and the condition $C=0$
is equivalent to
$$\kappa a\in\{\frac{\sigma^{2}}{4},\frac{3\sigma^{2}}{4}\}\,\, .$$
Let us analyze more closely the situations in which $C=0$.

\bf{1)} 

$\widetilde{\phi}=\displaystyle\frac{\alpha}{4}$\rm, \it i.e. \rm $\delta=1\, .$

Then $y(t)$ is a solution of
$$
dy(t)=\displaystyle\frac{\alpha}{2}dw(t)-\frac{\lambda}{2}y(t)dt \nonumber \,\, ;
$$
in particular, for $\lambda>0$, $y(t)$ is an Ornstein-Uhlenbeck process
(it was already known that the Ornstein--Uhlenbeck process was a Bernstein process for a quadratic potential).
Therefore when $\lambda > 0$, $z(t)$ co\"\i ncides, on the random interval $[0,T[$, with an Ornstein--Uhlenbeck
process.
Here 
$$
\eta(t,q)=e^{\displaystyle\frac{\lambda t}{4}-\displaystyle\frac{\lambda q^{2}}{\alpha^{2}}}\,\, .
$$
From
\begin{eqnarray}
y(t)\nonumber
&=&e^{-\frac{\lambda t}{2}}(y_{0}+\frac{\alpha}{2}\int_{0}^{t}e^{\frac{\lambda s}{2}}dw(s)) \nonumber \\
&=&e^{-\frac{\lambda t}{2}}(z_{0}+\widetilde{w}(\frac{\alpha^{2}(e^{\lambda t}-1)}{4\lambda})) \nonumber
\end{eqnarray}
($\widetilde{w}$ denoting another Brownian motion),
it appears that, for fixed $t$, $y(t)$ follows a normal law with mean $e^{-\frac{\lambda t}{2}}z_{0}$ and variance $\frac{\alpha^{2}(1-e^{-\lambda t})}{4\lambda}$.
The density $\rho_{t}(q)$ of $y(t)$ is therefore given by :
$$
\rho_{t}(q)=\displaystyle\frac{2\sqrt{\lambda}}{\alpha\sqrt{2\pi(1-e^{-\lambda t})}}
\exp{(-\displaystyle\frac{2\lambda(q-e^{-\frac{\lambda t}{2}}z_{0})^{2}}{\alpha^{2}(1-e^{-\lambda t})})}\,\, .
$$
Whence
\begin{eqnarray}
\forall t>0 \,\,\,\,
\eta_{*}(t,q) 
&=&\displaystyle\frac{\rho_{t}(q)}{\eta(t,q)} \nonumber \\
&=&\displaystyle\frac{1}{\alpha}\sqrt{\displaystyle\frac{\lambda}{\pi\sinh{(\displaystyle\frac{\lambda t}{2})}}}e^{(\displaystyle\frac{-\lambda q^{2}-\lambda q^{2}e^{-\lambda t}+4\lambda qz_{0}e^{-\frac{\lambda t}{2}}-2\lambda z_{0}^{2}e^{-\lambda t}}{\alpha^{2}(1-e^{-\lambda t})})} \nonumber
\end{eqnarray} 
and one may check that, as was to be expected, $\eta_{*}$ satisfies the following equation dual to Eq(\ref{eqdep}) :
 
\begin{eqnarray}
-\gamma\displaystyle\frac{\partial \eta_{*}}{\partial t}=-\displaystyle\frac{\gamma^2}{2}
\displaystyle\frac{\partial^{2}\eta_{*}}{\partial q^{2}}+V\eta_{*}          \,\, . \label{eqconj}
\end{eqnarray}

\bf{2)}$\widetilde{\phi}=\displaystyle\frac{3\alpha}{4}$\rm, \it i.e. \rm $\delta=3$.

In that case, according to \textit{Theorem \ref{theo1Aff1}}, $T=+\infty$ whence $y=z$.
Furthermore
$$
\eta(t,q)=q e^{\displaystyle\frac{\lambda}{\alpha^{2}}(\displaystyle\frac{3\alpha^{2}t}{4}-q^{2})}\,\, .
$$
Let us define
$$
s(t)=e^{-\displaystyle\frac{\lambda t}{2}}\frac{1}{z(t)}\,\, ;
$$
then an easy computation, using It\^o's formula in the same way as above, shows that
$$
ds(t)=-\frac{\alpha}{2}e^{\frac{\lambda t}{2}}s(t)^{2}dw(t)\,\, ;
$$
in particular, $s(t)$ is a martingale.

Referring once more to \textit{Proposition \ref{StrongSolu}} and its proof, we see that
$$
dX_{t}=\alpha \sqrt{X_{t}}dw(t)+(\frac{3\alpha^{2}}{4}-\lambda X_{t})dt\,\, .
$$

Let us now assume $X_{0}=0$ and $\lambda \neq 0$ ; then, according to \textit{Corollary \ref{BESQ}},
$$
X_{t}=e^{-\lambda t} Y(\frac{\alpha^{2}(e^{\lambda t}-1)}{4\lambda})
$$
where $Y$ is a $BESQ^{3}$(squared Bessel process with parameter 
$3$) such that $Y(0)=0$.
But, for each fixed $t>0$, $Y_{t}$ has the same law as $tY_{1}$, and $Y_{1}=\vert\vert B_{1}\vert\vert^{2}$
is the square of the norm of a $3$--dimensional Brownian motion ; the law of $Y_{1}$
is therefore
$$
\frac{1}{\sqrt{2\pi}}e^{-\frac{u}{2}}\sqrt{u}\mathbf 1_{u\geq 0}du\,\, .
$$
Therefore the density $\rho_{t}(q)$ of the law of $z(t)$ is given by :
$$
\rho_{t}(q)=\frac{1}{\sqrt{2\pi}}\frac{16\lambda^{\frac{3}{2}}}{\alpha^{3}(1-e^{-\lambda t})^{\frac{3}{2}}}q^{2}e^{-\displaystyle\frac{2\lambda q^{2}}{\alpha^{2}(1-e^{-\lambda t})}}
$$
and
$$
\forall t>0 \,\, \eta_{*}(t,q)=\frac{\rho_{t}(q)}{\eta(t,q)}=
\frac{16\lambda^{\frac{3}{2}}}{\alpha^{3}\sqrt{2\pi}}(1-e^{-\lambda t})^{-\frac{3}{2}}qe^{-\displaystyle\frac{3\lambda t}{4}-\displaystyle\frac{\lambda q^{2}}{\alpha^{2}\tanh(\frac{\lambda t}{2})}} \,\, .
$$

Here, too, one may check directly that $\eta_{*}$ satisfies Eq(\ref{eqconj}) above.\\

M.Houda \cite{theseHouda} has extended these computations.

This approach of symmetries for SDE has also been used for purposes other than stochastic finance (Cf \cite{AP} for instance).

\section{Conclusion and prospects}

It may seem strange to start from Cartan's geometrical ideas for the integrability of Hamilton-Jacobi-Bellman equation when the initial purpose is to solve some stochastic differential equations. In fact there are very good reasons for such an approach. The Wiener process involved in Eq(\ref{Intro.defz}) is itself, of course, indissociable from the free heat equation ($V=0$ in Eq(\ref{HJB.edpeta})). 
This heat equation does not carry any direct dynamical meaning except in analogy with the Schrödinger equation for the same Hamiltonian $H$. But the non-linear transformation Eq(\ref{HJB.defS}) changes the situation. Interestingly, it appeared for the first time, and in the other direction, in the historic publication of Schrödinger where he introduced the equation bearing his name (1926). In stochastic control theory it is often called Fleming's logarithmic transformation.\\
The resulting Hamilton-Jacobi-Bellman Eq(\ref{HJB.EDPS}) is the stochastic deformation of its classical counterpart, for the same system.
Its Laplacian term encodes the Brownian like properties of underlying trajectories. As such, the geometric study \enquote{à la Cartan} of HJB should provide the same dynamical informations as in classical mechanics, but along the appropriate stochastic deformation of classical paths, i.e. some diffusion processes.
As it is clear from the differential ideal $I=(\omega,\Omega, \beta)$ of \S 2, these underlying diffusions are entirely characterized in terms of some solutions of HJB. Moreover, the basis of a Lagrangian analysis (the form $\omega$) and an Hamiltonian one (the form $\Omega$) are included in the framework. This means that to study the conditions under which the ideal $I$ is invariant under isovectors dragging is to study at once all the dynamical symmetries of a very large class of diffusion processes associated with $H$. At the end, the method provides a collection of martingales of those processes which is the stochastic deformation of the class of first integrals of classical system. It is already known (cf \cite{LZ2}) that, with them, absolutely continuous transformations of diffusions can be made, a kind of quadrature of processes on the bases of their dynamics. In particular, starting from the geometric study of the free case, a large class of Hamiltonians with quadratic coefficients can be treated by the same token. It should be possible, however, to obtain a direct interpretation of those transformations (or \enquote{symplectic diffeomorphisms}) in extended phase space. This aspect should be considered in future works.\\
For a more complete overview of the program of stochastic deformation, which is closer to mathematical structures of quantum mechanics than what was mentioned here, one may consult \cite{R2}.

\end{document}